\documentclass[11pt, oneside]{article}   	
\usepackage[margin=1in]{geometry}              		
\usepackage{graphicx}				

\usepackage{amssymb}
\usepackage{amsthm}
\usepackage{amsmath}
\usepackage{hyperref}
\usepackage{xcolor}

\usepackage{float}
\usepackage{booktabs} 
\usepackage{flushend}
\usepackage{bm}

\usepackage[vcentermath]{youngtab}
\newcommand{\CC}{\mathbb{C}}

\newcommand{\E}{\mathbb{E}}
\newcommand{\V}{\mathbb{V}}
\renewcommand{\vec}{\boldsymbol}

\newcommand{\RR}{\mathbb{R}}
\newcommand{\cA}{\mathcal{A}}
\newcommand{\cK}{\mathcal{K}}
\newcommand{\cC}{\mathcal{C}}
\newcommand{\cS}{\mathcal{S}}
\newcommand{\QQ}{\mathbb{Q}}

\newcommand{\sL}{\operatorname{SL}}
\newcommand{\GL}{\operatorname{GL}}
\newcommand{\rank}{\operatorname{rank}}
\newcommand{\Tr}{\operatorname{Tr}}
\newcommand{\diag}{\operatorname{diag}}
\newcommand{\poly}{\operatorname{poly}}
\newcommand{\trun}{\operatorname{trun}}
\newcommand{\ds}{\operatorname{ds}}
\newcommand{\row}{\operatorname{row}}
\newcommand{\supp}{\operatorname{supp}}

\newcommand{\capacity}{\operatorname{cap}}
\newcommand{\Mat}{\operatorname{Mat}}

\floatstyle{boxed}\newfloat{Algorithm}{ht}{alg}
\newtheorem{theorem}{Theorem}
\newtheorem{lemma}[theorem]{Lemma}
\newtheorem{proposition}[theorem]{Proposition}
\newtheorem{corollary}[theorem]{Corollary}
\newtheorem{fact}[theorem]{Fact}

\theoremstyle{definition}
\newtheorem{qn}[theorem]{Question}
\newtheorem{definition}[theorem]{Definition}
\newtheorem{example}[theorem]{Example}
\newtheorem{remark}[theorem]{Remark}

\title{Operator Scaling with Specified Marginals}
\author{Cole Franks \thanks{Supported in part by Simons Foundation award 332622.} \thanks{Department of Mathematics, Rutgers University,
Hill Center, 110 Frelinghuysen Road Piscataway NJ 08854-8019 USA}}
\begin{document}
\maketitle


\begin{abstract}
The completely positive maps, a generalization of the nonnegative matrices, are a well-studied class of maps from $n\times n$ matrices to $m\times m$ matrices. The existence of the operator analogues of doubly stochastic scalings of matrices, the study of which is known as \emph{operator scaling}, is equivalent to a multitude of problems in computer science and mathematics such rational identity testing in non-commuting variables, noncommutative rank of symbolic matrices, and a basic problem in invariant theory (Garg et. al., 2016).\\
\indent We study \emph{operator scaling with specified marginals}, which is the operator analogue of scaling matrices to specified row and column sums (or marginals). We characterize the operators which can be scaled to given marginals, much in the spirit of the Gurvits' algorithmic characterization of the operators that can be scaled to doubly stochastic (Gurvits, 2004). Our algorithm, which is a modified version of Gurvits' algorithm, produces approximate scalings in time $\operatorname{poly}(n,m)$ whenever scalings exist. A central ingredient in our analysis is a reduction from operator scaling with specified marginals to operator scaling in the doubly stochastic setting.\\
\indent Instances of operator scaling with specified marginals arise in diverse areas of study such as the Brascamp-Lieb inequalities, communication complexity, eigenvalues of sums of Hermitian matrices, and quantum information theory. Some of the known theorems in these areas, several of which had no algorithmic proof, are straightforward consequences of our characterization theorem. For instance, we obtain a simple algorithm to find, when it exists, a tuple of Hermitian matrices with given spectra whose sum has a given spectrum. We also prove new theorems such as a generalization of Forster's theorem (Forster, 2002) concerning radial isotropic position.

\end{abstract}

%
%



\tableofcontents

\section{Introduction}


Completely positive maps are linear maps between spaces of matrices that, informally speaking, preserve positive-semidefiniteness in a strong sense. Completely positive maps generalize the nonnegative matrices and arise in quantum information theory \cite{Ja74}. To each completely positive map $T:\Mat_{n\times n}(\CC) \to \Mat_{m \times m} (\CC)$ is associated another completely positive map $T^*: \Mat_{m\times m}(\CC) \to \Mat_{n \times n} (\CC)$ known as the \emph{dual} of $T$. In analogy with the matrix case, say a completely positive map is \emph{doubly stochastic} if $n = m$ and $T(I_n) = I_n$ and $T^*(I_n) = I_n$. A \emph{scaling} $T'$ of a completely positive map $T$ by a pair of invertible linear maps $(g,h)$ is another completely positive map $T_{g,h}:X \mapsto g^\dagger T (h X h^\dagger) g$. One is led to ask which completely positive maps have doubly stochastic scalings; \emph{operator scaling} is the study of this question. In fact, several other problems such as rational identity testing in non-commuting variables \cite{GGOW16}, membership in the null-cone of the left-right action of $\sL_n(\CC)\times\sL_m(\CC)$ \cite{GGOWbl16}, and a special case of Edmonds' problem \cite{Gu04} each reduce to (or are equivalent to) an approximate version of this question. In \cite{Gu04}, Gurvits gave two useful equivalent conditions for approximate scalability: a completely positive map $T:\CC^n \to \CC^n$ can be approximately scaled to doubly stochastic if and only if $T$ is \emph{rank-nondecreasing}, i.e. $\rank T(X) \geq \rank X$ for all $X \succeq 0$, or equivalently $\capacity T  > 0$ where \begin{equation*}\capacity T:= \inf_{X \succ 0}\frac{\det T(X)}{ \det X}. \label{cap_formula}\end{equation*}
Gurvits also gave an algorithm to compute approximate scalings if either of these equivalent conditions hold. The authors of \cite{GGOW16}, \cite{GGOWbl16},\cite{Gu04} analyzed the same algorithm to obtain polynomial-time decision algorithms for each of the aforementioned problems.\\

We consider a natural generalization of doubly stochastic scalings. Say $T$ \emph{maps $(A\to B, C\to D)$} if $T(A) = B$ and $T^*(C) = D$.
Say $T$ is \emph{$(I_n \to Q, I_m \to P)$-scalable} if there exists a scaling of $T$ that maps $(I_n \to Q, I_m \to P)$.\footnote{ We could just as well have discussed $(A \to B, C \to D)$-scalability, but it is equivalent to approximate (resp. exact) scalability to $(I_n \to Q, I_m \to P)$-scalability for $P = A^{1/2} D A^{1/2}$ and $Q = C^{1/2} B C^{1/2}$.}

\begin{qn}\label{amcos_four_stoc}Given positive semidefinite matrices $P$ and $Q$, which completely positive maps are $(I_n \to Q, I_m \to P)$-scalable?\end{qn}

\indent Note that $(I_n \to I_n, I_n \to I_n)$-scalability is precisely the doubly stochastic case treated by Gurvits in \cite{Gu04}. We extend Gurvits' characterization of approximate scalability to the setting of Question \ref{amcos_four_stoc}. As in \cite{Gu04}, our proofs show how to efficiently produce approximate scalings when they exist. Our first main theorem, which closely resembles the characterization in \cite{Gu04}, characterizes the existence of approximate $(I_n \to Q, I_m \to P)$-scalability by block-upper-triangular matrices. Next we extend this characterization to handle scaling by the full group of invertible matrices with a somewhat surprising outcome - informally, a completely positive map $T$ is approximately $(I_n \to Q, I_m \to P)$-scalable if and only if a random scaling of $T$ is approximately $(I_n \to Q, I_m \to P)$-scalable \emph{by upper triangular matrices} with high probability.\\
\indent An interesting property of the characterizations is that for $T$ fixed, the admissible spectra of $P$ and $Q$ form a convex polytope known as a \emph{moment polytope}. This is a special case of a more general phenomenon obeyed by group actions studied in algebra, geometry and physics; see \cite{CDW13}. 

\subsection{Prior Work}
\indent Following a close variant of Question \ref{amcos_four_stoc} asked in \cite{GP15}, a special case of our characterization had already been proven in \cite{Fr16} via fixed-point theorems apparently unrelated to our techniques: namely, it was shown that any \emph{positivity-improving} operator, or operator that maps the nonzero positive-semidefinite matrices to the positive-definite matrices, is $(I_n \to Q, I_m \to P)$-scalable.\footnote{Given the necessary condition, which we will soon see is obvious, that $\Tr P = \Tr Q$.}\\
\indent The convexity of the marginals of scalings of $T$ follows from well-known theorems on moment polytopes; for a survey of these results see \cite{CDW13}. These techniques could also likely be used to deduce our characterization of $(I \to Q, I \to P)$-scalability \cite{Fz02}, but to the author's knowledge this has not been done explicitly in the literature. Further, these techniques do not result in algorithms and rely on powerful theorems from geometric invariant theory. In contrast, our proofs yield algorithms for computing approximate scalings and require only the results of \cite{Gu04},\cite{GGOW16} and a few basic facts from algebraic geometry. One similarity does exist: the so-called \emph{shifting trick} \cite{Bri87} in geometric invariant theory is a different type of reduction, but one that does not stay within the realm of operator scaling.
\subsection{Special Cases}\label{subsec:special_cases}
Here we list a few questions that reduce to or are special cases of Question \ref{amcos_four_stoc}. 
\begin{qn}[Matrix scaling]\label{qn:rc_scaling}
Given a nonnegative matrix $A \in \Mat_{m,n}(\RR)$ and nonnegative row- and column-sum vectors $\mathbf{r} \in \RR^m_{\geq 0}$ and $\mathbf{c} \in \RR^n_{\geq 0}$, do there exist diagonal matrices $X,Y$ such that the row (resp. column) sums of $A' = XAY$ are $\mathbf{r}$ (resp. $\mathbf{c}$)?
\end{qn}
It is well-known that matrix scaling can be reduced to an instance of operator scaling with specified marginals, but Gurvits' characterization for operator scaling does not apply to this instance unless $\mathbf{r}$ and $\mathbf{c}$ are the all-ones vectors. 
In Section \ref{rc_scaling_app}, we recall the reduction from Question \ref{qn:rc_scaling} to Question \ref{amcos_four_stoc} and derive the classic theorem of \cite{RS89} on the existence of such scalings as a consequence of Theorem \ref{thm:apx_delta_scalable}.
\begin{qn}[Eigenvalues of sums of Hermitian matrices]\label{hermitians}
Given nonincreasing sequences $\alpha, \beta, \gamma$ of $m$ real numbers, are $\alpha, \beta, \gamma$ the spectra of some $m\times m$ Hermitian matrices $A,B,C$ satisfying 
$$A + B = C?$$
\end{qn}
In \cite{Kl98}, Klyachko showed (amazingly) that the answer to Question \ref{hermitians} is ``yes" if and only if $\alpha, \beta, \gamma$ satisfy a certain finite set $S_m$ of linear constraints. That is, such $(\alpha, \beta, \gamma)$ form a polyhedral cone. A long line of work has been devoted to describing the set $S_m$, which has connections to representation theory, Schubert calculus, and combinatorics \cite{KT00}, \cite{Kl98}, \cite{F00}. There are even polynomial-time algorithms to test if $\alpha, \beta, \gamma$ satisfy $S_m$ \cite{MNS12}. However, no previous work has provided an algorithm to \emph{find} the Hermitian matrices..\\
\indent In Section \ref{hermitians_app}, we show that Question \ref{hermitians} can also be reduced to Question \ref{amcos_four_stoc}.
Our reduction yields an algorithmic proof of Klyachko's characterization in \cite{Kl98}; see Algorithm \ref{alg:inverse_eig}. Question \ref{hermitians} can be reduced by shifting by multiples of the identity and scalar multiplication\footnote{This is also why one can assume $\alpha_m, \beta_m, \gamma_m \geq 1/4$ without loss of generality in Algorithm \ref{alg:inverse_eig}.} to the following: given nonincreasing sequences $\alpha, \beta, \gamma$ of $m$ real numbers, are $\alpha, \beta, \gamma$ the spectra of some $m\times m$ Hermitian matrices
$A + B + C = I_m$? Algorithm \ref{alg:inverse_eig} approximates such matrices, if they exist:
\begin{proposition} \label{prp:inverse_eig}
Algorithm \ref{alg:inverse_eig} runs in time $O( b m^2/\epsilon^2)$ and outputs ERROR with probability at most $1/3$ if $\alpha, \beta, \gamma$ are of total bit-complexity $b$ and are the spectra of some $m\times m$ positive-semidefinite matrices $A, B, C$ satisfying $A + B + C = I_m$.
\end{proposition}
The correctness of Algorithm \ref{alg:inverse_eig} follow from Proposition \ref{horn_red} in Section \ref{hermitians_app} and the correctness of Algorithm \ref{alg:informal_scaling}.

\begin{Algorithm}[th!]
\begin{flushleft}
\textbf{Input:} Nonincreasing length-$m$ sequences $\alpha, \beta, \gamma$ of rational numbers in $(1/4, 1]$ whose total sum is $m$.\\
 \vspace{.25cm}
\textbf{Output:} Matrices $A, B, C$ with $\lambda(A) = \alpha, \lambda(B) = \beta, \lambda(C) = \gamma$ and 
$$\|A + B + C - I_m\| \leq \epsilon.$$ 
\textbf{Algorithm:}
\begin{enumerate}
\item 
Choose each entry of $U_A, U_B, U_C \in  \Mat_{m\times m}(\CC)$ uniformly at random from $[6 \cdot 2^b]$, where $b$ is the total bit-complexity of the input. If one of $U_A, U_B,$ or $U_C$ is singular, \textbf{return} ERROR. 

\item Repeat; if any step is not possible, \textbf{return} ERROR.
\begin{enumerate}
\item Choose $a, b, c$ upper triangular such that $U_A a, U_B b, U_C c$ are unitary. Set $U_A \leftarrow U_A a$, $U_B \leftarrow U_B b$, and $U_C \leftarrow U_C c$.
\item\begin{description}
\item [If:] $$\|U_A \diag(\alpha) U_A^\dagger + U_B \diag(\beta) U_B^\dagger + U_C \diag(\gamma) U_C^\dagger\| \leq \epsilon,$$
\textbf{return} $A = U_A \diag(\alpha) U_A^\dagger$, $B = U_B \diag(\beta) U_B^\dagger$, and $C = U_C \diag(\gamma) U_C^\dagger$.\\
\item [Else:] Choose $g$ lower-triangular such that 
$$ g\left(U_A \diag(\alpha) U_A^\dagger + U_B \diag(\beta) U_B^\dagger + U_C \diag(\gamma) U_C^\dagger\right) g^\dagger = I_m$$
and set $U_A \leftarrow gU_A $, $U_B \leftarrow gU_B $, and $U_C \leftarrow gU_C $.
\end{description}
\end{enumerate}
\end{enumerate}
\caption{Algorithm for Question \ref{hermitians}.}\label{alg:inverse_eig}
\end{flushleft}
\end{Algorithm}

\begin{qn}[Forster's scalings]\label{forst_quest_stoc}
Given vectors $u_1, \dots, u_n \in \CC^m$, nonnegative numbers $p_1, \dots, p_n$, and a positive-semidefinite matrix $Q$, when does there exist an invertible linear transformation $B : \CC^m \to \CC^m$ such that 
$$\sum_{i =1}^{n}p_i\frac{  B u_i  (B u_i)^\dagger }{\|Bu_i\|^2}  = Q?$$
\end{qn}
Forster answered Question \ref{forst_quest_stoc} in the positive for $p_i =1$, $u_i$ in general position, and $Q = \frac{n}{m}I_m$; as a consequence he was able to prove previously unattainable lower bounds in communication complexity \cite{Fo02}. As noted in \cite{Gu04}, Forster's result is a consequence of Gurvits' characterization of doubly stochastic scalings. Independently, Barthe \cite{B98} answered this question completely for the case $Q = I_m$ in order to study the rank-one Brascamp-Lieb inequalities. He showed the left hand side can approach the right hand side if and only if $(p_1, \dots, p_n)$ lies in the \emph{basis polytope} of $u_1, \dots, u_n$. In Section \ref{forst_app} we reduce the general case of Question \ref{forst_quest_stoc} to an instance of Question \ref{amcos_four_stoc}, and use this reduction to answer the approximate version of Question \ref{forst_quest_stoc}. For fixed $u_1, \dots, u_n$ and $Q$, the admissible $(p_1, \dots, p_n)$ form a convex polytope with known as an \emph{polymatroid}, of which the basis polytope is a special case. It would be interesting to find an application of our more general characterization, perhaps along the lines of the use of Barthe's theorem in subspace recovery in \cite{AM13}.
\begin{qn}[Quantum marginals]\label{marginals}For which positive-semidefinite operators $\rho:\CC^n \otimes \CC^m \to \CC^n \otimes \CC^m$ and positive-semidefinite operators $P:\CC^n \to \CC^n$, $Q:\CC^m \to \CC^m$ can the reduced density matrices of $(g\otimes h) \rho (g \otimes h)^\dagger$ be made arbitrarily close to $P$ and $Q$?\end{qn}
Question \ref{marginals} is equivalent to (the approximate version of) Question \ref{amcos_four_stoc} by a correspondence known as \emph{state-channel duality} \cite{Ja74}. The convex polytope of admissible spectra for the tripartite version of Question \ref{marginals}, in which there are three specified reduced density matrices $P, Q$ and $R$ rather than two, is is called the \emph{Kronecker polytope} and arises in the representation theory of the symmetric group. A polynomial time algorithm for Question \ref{amcos_four_stoc} would confirm that membership in the Kronecker polytope is in $\mathbf{P}$, whereas it is only known to be in $\mathbf{coNP} \cap \mathbf{NP}$ \cite{BCMW17}. The author views this paper as a step towards solving this problem - however, we have only two specified marginals and can produce only a very inefficient membership oracle for our polytope. The work \cite{BFGOWW18}, which came shortly after this one, remedies the former issue but not the latter.
\subsection{Organization of the Paper}
\begin{itemize}
\item In Section \ref{results_sec} we describe some background and state our main theorems, Theorem \ref{thm:apx_delta_scalable}, Theorem \ref{thm:gln_scalable}, and Theorem \ref{thm:gln_alg}. We also describe a few of the techniques that are needed in the proofs, and how they differ from previous work.
\item In Section \ref{sec:reduction} we define a reduction to the doubly stochastic case on which most of our proofs rely.
\item In Section \ref{sec:triangular_scalings} we prove our main theorem characterizing scalability by upper triangulars, Theorem \ref{thm:apx_delta_scalable}. Sections \ref{sec:reduction} and \ref{sec:triangular_scalings} contain elements of this proof; the elements are assembled in \ref{subsec:apx_delta_scalable_proof} We also analyze the running time of Algorithm \ref{alg:alg_g_tri} for finding upper triangular scalings when they exist.
\item In Section \ref{sec:general_scalings} we extend the results of Section \ref{sec:triangular_scalings} to prove Theorem \ref{thm:gln_scalable}, our characterization of scalability by the full general-linear groups. We also analyze Algorithm \ref{alg:informal_scaling} for finding \emph{general} scalings, when they exist.
\item In Section \ref{sec:applications} we extend our main theorems to characterize scalability by direct sums of general linear groups. We then use this extension to reduce Questions \ref{qn:rc_scaling}, \ref{hermitians}, and \ref{forst_quest_stoc} to operator scaling with specified marginals. 
\end{itemize}
\section{Background and results}\label{results_sec}
\subsection{Preliminaries}\label{subsec:prelim}

We will require a few definitions. If $A$ is a matrix, $\| A \|:=\sqrt{\Tr A A^\dagger}$. \begin{definition}[completely positive maps]A completely positive map is a map $T: \Mat_{n\times n}(\CC)\to \Mat_{m\times m}(\CC)$ of the form 
$$T:X \mapsto \sum_{i = 1}^r A_i X A_i^\dagger,$$
 $A_i:\CC^n \to \CC^m$ are linear maps called \emph{Kraus operators} of $T$. Note that $T$ preserves positive-semidefiniteness. The map $T^*:\Mat_{m\times m}(\CC)\to \Mat_{n\times n}(\CC)$ is given by 
$$T^*:X \mapsto \sum_{i = 1}^r A_i^\dagger X A_i,$$ and is the adjoint of $T$ in the trace inner product $\langle A, B \rangle = \Tr A^\dagger B$, where $A^\dagger$ is the conjugate transpose of $A$, where any $\CC^d$ is understood to be equipped with the standard Hermitian inner product. 
\end{definition}

\begin{definition}[approximate scalings]Say a scaling $T'$ of $T$ is an $\epsilon$-$(A\to B, C\to D)$-\emph{scaling} if $T'$ maps $(A \to B', C\to D')$ with $\|B - B'\| \leq \epsilon$ and $\|D - D'\| \leq \epsilon$. If $G \subset \GL_n(\CC)$ and $H \subset \GL_m(\CC)$, say $T$ is \emph{approximately $G\times H$-scalable to} $(A\to B, C\to D)$ if for all $\epsilon > 0$, $T$ has an $\epsilon$-$(A\to B, C\to D)$-scaling $T'$ by $(g,h) \in G \times H$.
\end{definition}
Throughout, we'll let $E = (e_1, \dots, e_n)$ and $F = (f_1, \dots, f_m)$ be the standard orthonormal bases for $\CC^n$ and $\CC^m$. We'll assume that 
$$P = \diag(p_1, \dots, p_n) \textrm{ and } Q = \diag(q_1, \dots, q_n),$$
where $\mathbf{p} = (p_1, \dots, p_n)$ and $\mathbf{q} = \diag(q_1, \dots, q_n)$ are non-increasing sequences of nonnegative real numbers.\\
\indent This is without loss of generality: the groups $G$ and $H$ of interest will contain the diagonalizing unitaries for anything the image of $T$ and $T^*$, respectively, and if $T_{g,h}$ is an $(I \to Q, I \to P)$-scaling of $T$ then $T_{gV,hU}$ is an $(I_n \to V^\dagger Q V, I_m \to U^\dagger P U)$-scaling for any unitaries $U$ and $V$.
If $\mathbf{a} = (a_1, \dots, a_k)$ is a non-increasing sequence of real numbers, the shorthand
$$ \Delta a_i = a_i - a_{i +1} \textrm{ where } a_{k + 1}:=0$$
will also be helpful.
\begin{definition}[Flags]
We'll first consider scalings by upper triangular matrices. Equivalently, the scalings must preserve \emph{the standard flag} 
$$E_\bullet = ( \;\langle e_1 \rangle, \langle e_1, e_2 \rangle, \dots, \langle e_1, \dots, e_{n-1} \rangle \; )$$
in the basis $E$ and the standard flag $F_\bullet$ in the basis $F$. More generally, a \emph{flag} is an increasing sequence of subspaces. 
\end{definition}


\subsection{Capacity and Rank-Nondecreasingness Extended}
Our results split into two parts: a characterization of scalability by upper-triangular matrices,  and the application of this characterization to $T_{g,h}$ for generic $g$ and $h$. Our description of scalability by upper-triangular matrices is very similar to Gurvits' characterization. We first need analogues of rank-nondecreasingness and capacity.
\subsubsection*{Rank-nondecreasingness}
\indent Recall that a completely positive operator $T:\CC^n \to \CC^n$ is rank-nondecreasing if $\rank T(X) \geq \rank X$ for all $X \succeq 0$. Rank-nondecreasingness is much like a Hall's condition for completely positive operators. We restate the definition using a notion we call \emph{$T$-independence}, which is motivated by independent sets in bipartite graphs.
\begin{definition}[$T$-independence]\label{def:t_indep}
Suppose $T$ is as in \ref{subsec:prelim}. We say a pair of subspaces $(L \subset \CC^m, R \subset \CC^n)$ is \emph{$T$-independent} if 
$$L \subset (A_i R)^\perp\textrm{ for all } i \in [r].$$
Equivalently, the pair $(L,R)$ is $T$-independent if and only if $\pi_L T(\pi_R) = 0$ where $\pi_L, \pi_R$ are the orthogonal projections to the subspaces $L$ and $R$. 
\end{definition}

In this language, it's not hard to see that $T$ is rank-nondecreasing if and only if $\rank L + \rank R \leq n$ for all $T$-independent pairs $(L,R)$ - analogous to the fact that a bipartite graph $G$ on $[n] \cup [n]$ satisfies Hall's condition if and only if the cardinality of the largest independent set is at most $n$. We extend the definition as follows: 
\begin{definition}[rank-nondecreasingness for specified marginals]\label{def:rnd_diff} Suppose $T, P, Q$ are as in \ref{subsec:prelim}. Say $T$ is \emph{$(P,Q)$-rank-nondecreasing} if for all $T$-independent pairs $(L,R),$
\begin{equation}\sum_{i =1}^m \Delta q_i \dim E_i \cap L + \sum_{j =1}^n \Delta p_j \dim F_j \cap R \leq \Tr P.\label{diff_rnd_eq}\end{equation}
\end{definition}

\begin{remark}\label{rem:schubert}
The purpose of this remark is to relate our rank-nondecreasingness condition with some familiar mathematical objects; namely the Schubert cells and varieties; see the introduction \cite{Le10}. The $k$-dimensional subspaces of $\CC^m$ satisfying $d_i = \dim E_i \cap L$ for all $i \in [m]$ for a fixed sequence $\textbf d$ is known as the \emph{Schubert cell} $\Omega^\circ_I(E_\bullet)$, where $I = \{i: d_i - d_{i - 1} \cap L > 0\} \subset [m]$ is the sequence of ``jumps" in the dimension. Observe that if $ L \in \Omega^\circ_I(E_\bullet)$ we have 
$$\sum_{i =1}^m \Delta q_i \dim E_i \cap L = \sum_{i \in I} q_i.$$ The Schubert cells partition the collection of $k$-dimensional subspaces of $\CC^m$ (the Grassmanian $\operatorname{Gr(n,k)}$). The \emph{Schubert variety} $\Omega^\bullet_I(E_\bullet)$ is the Zariski closure of $\Omega^\circ_I(E_\bullet)$ in $\operatorname{Gr}(m,k)$, and can be equivalently described as the $k$-dimensional subspaces such that for all $j \in [k]$, $\dim E_{i_j} \cap L \geq j$ where $I = \{i_1, \dots, i_k\}$. Further, if $\textbf p$ is monotone decreasing then the function $\sum_{i \in I} p_i$ is monotone decreasing in the lattice of Schubert varieties ordered by inclusion. All this goes to show that $T$ is $(P,Q)$-rank-nondecreasing if and only if 
\begin{equation}\sum_{i \in I} q_i + \sum_{j \in J} p_i \leq \Tr P \label{rnd_3}\end{equation}
for all $I \subset [m]$, $J \subset [n]$ such that $\Omega_I^\bullet (E_\bullet)\times \Omega_{J}^\bullet(F_\bullet)$ contains a $T$-independent pair. This formulation also enjoys a similarity with Rothblum and Schneider's condition for approximate scalability of matrices to specified row and column sums \cite{RS89}.  
\end{remark}
\subsubsection*{Capacity}
Next we need to extend the capacity 
$$ \capacity(T) = \inf_{X \succ 0}\frac{\det T(X)}{ \det X}.$$
For those familiar with the matrix-scaling case, we are motivated by how the capacity 
\begin{align}
 \inf_{x_j > 0}\frac{\prod_{i = 1}^m (A\textbf{x})_i^{r_i}}{\prod_{j = 1}^n x_j^{c_j}} \label{eq:class_cap}
\end{align}
extends the $\textbf{r}=\textbf{1}, \textbf{c} = \textbf{1}$ case. The numerator here is hyperbolic, so the nonvanishing of capacity tells us that the uniform monomial must be present and so the support of this matrix contains a perfect matching. \\
\indent We'll need a function extending the denominator $\prod_{j = 1}^n x_j^{c_j}$ to non-diagonal matrices. 
\begin{definition}[Relative determinant]\label{def:rel_det}
Let $\eta_j:\CC^k \to F_j$ be the coordinate projection to the first $j$ coordinates. The dimension $k$ will be clear from context.
\begin{align} \begin{array}{ccc} & n  & \\
\eta_j =& \left[ \begin{array}{ccc ccc} 1 & \hdots & 0 &0 &  \hdots & 0 \\
\vdots & \ddots & \vdots & \vdots & \ddots & \vdots \\
0 & \hdots & 1 & 0 &  \hdots & 0 \\
 \end{array}\right] & j.\end{array}\label{eq:projections}
 \end{align}
If $A:\CC^k \to \CC^k$ is a diagonal, positive-definite operator with spectrum $\mathbf{a} = (a_1, \dots, a_n)$ and $X \in \Mat_{k \times k} (\CC^m)$, define the \emph{determinant of $X$ relative to $A$}, denoted $\det(A, X)$, by 
\begin{align}
\det(A, X) = \prod_{j = 1}^n (\det \eta_j X \eta_j^\dagger)^{\Delta a_j }.\label{eq:rel_det}
\end{align}
We always use the convention $0^0 = 1$. 
\end{definition}
It is instructive to see how Eq. \ref{eq:rel_det} recovers the denominator of Equation \ref{eq:class_cap} when $X = \diag(\textbf{x})$ and $\textbf{a} = \textbf{c}$.

\begin{remark}
The relative determinant defined above is related to representation theory. A \emph{highest weight vector} of weight $\textbf a = (a_1 \geq \dots \geq  a_m) \in \RR^m$ in a representation $V$ of $\GL(\CC^m)$ is a vector $v \in V$ such that any upper triangular matrix $g$ acts on $v$ by the scalar multiplication $b\cdot v  = \chi_\textbf a(g)v$, where 
$$ \chi_\textbf a(g):= \prod_{i = 1}^m g_{ii}^{a_i}.$$
Observe that $\chi_\textbf a$ is multiplicative, i.e. $\chi_\textbf a(g_1 g_2) = \chi_\textbf a(g_1) \chi_\textbf a(g_2)$ for $g_1,g_2 \in \GL(E_\bullet)$. The relationship to $\det(A, X)$ is as follows: if $A = \diag(\textbf a)$ and $g_1, g_2$ are upper triangular, then 
\begin{align} \det(A, g_1^\dagger g_2) = \overline{\chi_\textbf a(g_1)}\chi_\textbf a(g_2). \label{eq:character}
\end{align}
Prove Eq. \ref{eq:character} via straightforward calculation in Appendix \ref{app:results_sec}. 
\end{remark}

Eq. \ref{eq:character} and the density of $L U$-factorable matrices in the space of square matrices immediately yields the following formulae:

\begin{lemma}[Properties of $\det(A,X)$]\label{detpx_facts}
If $h$ is upper triangular, then
\begin{align}
\det(A, X h) &=  \det(A, X) \det(A,  h),\label{oneside_mult}\\
\det(A, h^\dagger X h) &= \det(A, h^\dagger h) \det(A, X),\label{twoside_mult}\\
\textrm{ \emph{and} }\det(A, h^{-\dagger} h^{-1}) &= \det(A, h^\dagger h)^{-1}.\label{twoside_inv}
\end{align}
\end{lemma}

We can now define the capacity:
\begin{definition}[Capacity for specified marginals]\label{def:capacity}
Define
\begin{align} 
 \capacity(T, P, Q) =  \inf_{h \in \GL(F_\bullet)} \frac{\det(Q, T(h P h^\dagger))}{\det(P, h^\dagger h)}.\label{eq:capacity}
\end{align}

\end{definition}

\subsection{Main Theorems}

We are ready to state our analogue of Gurvits' characterization for upper-triangular scalings. Recall Gurvits' theorem:
\begin{theorem}[Gurvits \cite{Gu04}]\label{thm:gurvits}
Let $T: \Mat_{n\times n} \CC \to \Mat_{n\times n} \CC$ be a completely positive map. The following are equivalent:
\begin{enumerate}
\item \label{delt_rnd}$T$ is rank-nondecreasing.
\item \label{delt_cap}$\capacity T > 0$. 
\item\label{delt_scal} $T$ is approximately $\GL(E_\bullet) \times \GL(F_\bullet)$-scalable to 
$$(I_n \to I_n, I_n \to I_n).$$
\end{enumerate}
\end{theorem}
The $QR$-decomposition shows that in the setting of the above theorem there is no loss of generality in scaling by $\GL(E_\bullet)\times \GL(F_\bullet)$ rather than $\GL_m(\CC) \times \GL_n(\CC)$. This is no longer true in our setting, which is why we need separate theorems for upper-triangular scalings and general scalings.
\begin{theorem}[Main theorem for upper-triangular scalings]\label{thm:apx_delta_scalable}
Let $T, P, Q$ be as in \ref{subsec:prelim}. The following are equivalent:
\begin{enumerate}
\item \label{delt_rnd}$T$ is $(P,Q)$-rank-nondecreasing.
\item \label{delt_cap}$\capacity(T, P, Q) > 0$. 
\item\label{delt_scal} $T$ is approximately $\GL(E_\bullet) \times \GL(F_\bullet)$-scalable to 
$$(I_n \to Q, I_m \to P).$$
\end{enumerate}
\end{theorem}

\indent Next we show how to characterize $\GL_m(\CC) \times \GL_n(\CC)$-scalability. It is enough to first scale by a generic element of $\Mat_m(\CC) \times \Mat_n(\CC))$ and then search for scalings to the target by elements of $\GL(E_\bullet)\times \GL(F_\bullet)$. \emph{Generic} means ``for all but those in an affine variety\footnote{set of common zeroes of a family of polynomials} that does contain $\Mat_m(\CC) \times \Mat_n(\CC))$."

\begin{theorem}[Main theorem for general scalings]\label{thm:gln_scalable}
Let $T, P, Q$ be as in \ref{subsec:prelim}. The following are equivalent:
\begin{enumerate}
\item \label{gen_p_q} $T_{g,h}$ is $(P,Q)$-rank-nondecreasing for generic 
$(g^\dagger, h) \in \GL_m(\CC) \times \GL_n(\CC).$
\item \label{gen_cap} $\capacity(T_{g,h}, P, Q) > 0$ for generic
$(g^\dagger, h) \in \GL_m(\CC) \times \GL_n(\CC).$
\item \label{gen_scal} $T$ is approximately scalable to $(I_n \to Q, I_n \to P).$

\end{enumerate}
\end{theorem}

The set $\{(\mathbf{p}, \mathbf{q}): T\textrm{ is }(P,Q)\textrm{-rank-nondecreasing}\},$ which we denote $\cK(T, E_\bullet, F_\bullet)$, is a convex polytope since it is defined by a finite number of linear constraints. Less obviously, the set of $(\mathbf{p}, \mathbf{q})$ such that $T$ is approximately $(\GL_m(\CC),\GL_n(\CC)$-scalable to $(I_m \to Q, I_n \to P)$ also forms a convex polytope, which we denote $\cK(T)$. This will follow from the proof of Theorem \ref{thm:gln_scalable}.\\
\indent We also prove an algorithmic counterpart of 
Theorem \ref{thm:gln_scalable}. 
\begin{definition}[bit-complexity]\label{dfn:bit_complexity} 
 Say $T, P, Q$ have \emph{bit-complexity at most $b$} if
 \begin{enumerate}
 \item the entries of the matrices $A_1, \dots, A_r$ are $a + bi$ where $a, b$ are binary,
 \item the entries of the vectors $\mathbf{p}$, and $\mathbf{q}$ are binary positive numbers which individually sum to one, and 
 \item the sum total of the number of digits from all of the entries above plus $\log_2 r + \log_2 m + \log_2 n$ is at most $b$.
 \end{enumerate}
\end{definition}


\begin{theorem}[Scaling algorithm]\label{thm:gln_alg} There is a randomized algorithm $\cA$ of time-complexity\\ $\poly(\epsilon^{-1}, p_n^{-1}, q_m^{-1}, b)$ that takes as input $T, P, Q$ of bit-complexity at most $b$ and $\epsilon> 0$ and outputs either ERROR or $(g, h) \in \GL_m(\CC)\times\GL_n(\CC)$ such that $T_{g,h}$ is an $\epsilon$-$(I_n \to Q,I_m \to P)$-scaling;\\
\indent  If $T$ is approximately $\GL_m(\CC) \times \GL_n(\CC)$-scalable to $(I_n \to Q, I_m \to P)$ then $\cA$ outputs ERROR with probability at most $1/3$.\end{theorem}

Our algorithm will actually find $\epsilon$-$(P \to I_m, Q \to I_n)$ scalings of $T$, but we will see that this is equivalent when the scalings are nonsingular. The nonsingularity of the scalings is also without loss of generality; we will also show how to reduce to the case in which the marginals are nonsingular. Algorithm \ref{alg:informal_scaling} is not exactly $\cA$ of Theorem \ref{thm:gln_alg}; rounding between scaling steps is required to avoid a blow-up in bit-complexity. See Remark \ref{rem:bit_complexity} for a discussion of these issues.
\begin{Algorithm}[th!]
\begin{flushleft}
\textbf{Input:} $T, P, Q$ of bit-complexity at most $b$.\\
 \vspace{.25cm}
\textbf{Output:} Either a pair $(g,h)$ such that $T_{g,h}$ is an $\epsilon$-$(P \to I_m, Q \to I_n)$ scaling of $T$, or ERROR. \\
 \vspace{.25cm}
\textbf{Algorithm:}
\begin{enumerate}
\item 
Choose each entry of $(g_0, h_0) \in \Mat_{m\times m}(\CC)\times \Mat_{n \times n}(\CC)$ independently and uniformly at random from $[6 \cdot 2^{b}]$. If $g_0$ or $h_0$ is singular, \textbf{return} ERROR.
\item For $j \in [TIME = \poly(\epsilon^{-1}, p_n^{-1}, q_m^{-1}, b)]$:
\begin{enumerate}
\item 
\begin{description}
\item[If $j$ is odd:] Find $g\in \GL({E_\bullet})$ such that $ g^\dagger T (h_{j-1}Ph_{j-1}^\dagger)g  = I$. Set $g_j = g$ and $h_j = h_{j-1}$.
\item[If $j$ is even:] Find $h \in \GL({F_\bullet})$ such that $ h^\dagger T^* (g_{j-1}Qg_{j-1}^\dagger)h  = I$. Set $g_{j} = g_{j-1}$ and  $h_j = h$.
\end{description}
If this was not possible, \textbf{return} ERROR.
\item If $ T_{g_j,h_j}$ is an $\epsilon$-$(P \to I_m, Q \to I_n)$ scaling of $T$,\\
 \textbf{return} $(g_j, h_j)$. 
\end{enumerate}
\item \textbf{Return} ERROR.
\end{enumerate}
\caption{Scaling algorithm.}\label{alg:informal_scaling}
\end{flushleft}
\end{Algorithm}

\subsection{Techniques}
Here we list the main technical issues that arise and how we overcome them.
\subsubsection*{\textbf{Reduction to the Doubly Stochastic Case.}}
The main difficulty in proving Theorem \ref{thm:apx_delta_scalable}, the characterization of upper-triangular scalability, is in guessing the correct notions of capacity and rank-nondecreasingness. In fact, they were not guessed directly, but rather follow from the existence of a reduction to the doubly stochastic case. The reduction, which is only defined when $P, Q$ have integral spectra, is a map
$$ T \mapsto \trun_{P,Q} T,$$
where $\trun_{P,Q} T:\Mat_{\Tr P, \Tr P} \to \Mat_{\Tr Q, \Tr Q}$ is yet another completely  positive map. $\trun_{P,Q}$ has the property that $T$ is scalable to $(I_n \to Q, I_m \to P)$ if and only if  $\trun_{P,Q} T$ can be scaled to doubly stochastic.\\
\indent One might be tempted to prove our main theorems by applying the results from \cite{Gu04} directly to $\trun_{P,Q} T$. However, this has two drawbacks: firstly, the algorithms are only guaranteed to run in polynomial time if the spectra of $P$ and $Q$ are represented in \emph{unary}. Secondly, the reduction does not work for irrational spectra, and so separate reasoning is necessary to extend Theorem \ref{thm:apx_delta_scalable} to that case.
\subsubsection*{\textbf{Running Time.}}

After an initial random scaling step, our algorithm is a natural variant of the Sinkhorn-Knopp algorithm for matrix scaling \cite{Si64}. This algorithm alternately scales so that $T(P) = I$ or $T^*(Q) = I$, which are individually easy to enforce. Each step will increase $\capacity(T, P, Q)$ unless $T$ is very close to mapping $(P \to I_m, Q \to I_n)$. While this proves that nonzero capacity implies approximate $(P \to I_m, Q \to I_n)$-scalability, it will not give any upper bound on the number of scaling steps unless we have a lower bound on the capacity.\\
\indent The lower bound ontained by computing the capacity lower bounds from \cite{GGOW16} on $\trun_{P,Q} T$ is not sufficient to prove Theorem \ref{thm:gln_alg}, which asserts the existence of a polynomial time algorithm for scaling. This lower bound relies on bounds on the degrees of polynomials invariant under scaling \cite{DM17}. However, the required degree bound has since been improved - surprisingly, this seemingly minor improvement is enough. We are able to make a further improvement using the fact that $e^{H(\textbf{p})}\capacity(T, P, Q)$ is \emph{log-concave in $\textbf{p}$ and $\textbf{q}$}. This results in a lower bound for $\capacity(T, P, Q)$ which depends only on $T, m, $ and $n$ subject to $\Tr P = \Tr Q = 1$. 

\subsubsection*{\textbf{General Scalings from Triangular Scalings}} The definition of the reduction $\trun_{P,Q}$, the proofs of its properties, and the proof of Theorem \ref{thm:apx_delta_scalable} require only the results of \cite{Gu04}, elementary linear algebra, and calculus.\\
\indent Theorem \ref{thm:gln_scalable} and Theorem \ref{thm:gln_alg}, however, require some algebraic geometry. Both use that $T_{g,h}$ fails to be $(P,Q)$-rank-nondecreasing if and only if $(g^\dagger,h)$ are in some affine variety depending only on $P$ and $Q$. When the spectra of $P$ and $Q$ are rational the algebraic geometry is elementary\footnote{only after relying on \cite{DM17} to find the polynomials!}. When the spectra are irrational we require Chevalley's Theorem on quantifier elimination in algebra (see \cite{M99}).



\section{Reduction to the doubly stochastic case}\label{sec:reduction}
In this section we prove Theorem \ref{thm:stoc_reds} below. The reduction, a map $T \mapsto \trun_{P,Q} T$, is inspired by the reduction between instances of the Brascamp-Lieb inequality in \cite{GGOWbl16}. Item 1 will be obvious from our construction. We only prove Items 2 and 3 in this section; we prove Item 4 in Appendix \ref{app:reduction}. The attentive reader will notice that we do not use Item 2 directly in the proof of Theorem \ref{thm:apx_delta_scalable}, but we include the proof anyway to motivate our construction.
 
\begin{theorem}\label{thm:stoc_reds}
Suppose $\mathbf{p}$ and $\mathbf{q}$ are integral and that the sum of each is $N$. There exists a map $T \mapsto \trun_{P,Q} T$ such that 
\begin{enumerate}
\item $\trun_{P,Q} T: \Mat_{N \times N} \CC \to \Mat_{N \times N} \CC$ is a completely positive operator.
\item If, in addition, $\textbf p $ and $\textbf q$ are positive, \label{stoc_reds:item:scaling} $\trun_{P,Q} T$ is approximately $\GL_N(\CC) \times \GL_N(\CC)$-scalable to $(I_N \to I_N, I_N \to I_N)$ if and only if $T$ is approximately $\GL(E_\bullet)\times \GL(F_\bullet)$-scalable to $(I_n \to Q, I_m \to P)$.
\item \label{stoc_reds:item:capacity} $\capacity \trun_{P,Q} T = \capacity(T, P, Q)$, and 
\item $\trun_{P,Q} T $ is rank-nondecreasing if and only if $T$ is $(P,Q)$-rank-nondecreasing.
\end{enumerate}
\end{theorem}
 We first design a ``gadget" to compose with $T$ to enforce the marginal conditions of $\trun_{P,Q} T$.  Recall that a partition $\lambda$ of a nonnegative integer $l$ with $k$ parts is a weakly decreasing sequence $(\lambda_1, \dots, \lambda_k)$ of nonnegative integers summing to $l$.
 
\begin{lemma}\label{lem:reduction_gadget}
For any partition $\lambda = (\lambda_1, \dots, \lambda_k)$ of $l$ with $\lambda_k > 0$, there is an injective, completely positive map $ G_\lambda:\Mat_{k\times k} \CC \to \Mat_{l \times l} \CC$ such that 
\begin{align}G_\lambda(I_k) = I_l \textrm{ and }G_\lambda^*(I_l) = \diag(\lambda).\label{eq:gadget_marg} \end{align}

$G_\lambda$ further satisfies 
\begin{align}G_\lambda(Xh) = G_\lambda(X)G_\lambda(h)\label{eq:gadget_hom}
\end{align}
 for any $X \in \Mat_{k\times k}(\CC)$ and any upper-triangular $h \in \Mat_{k\times k}(\CC),$ and 
 \begin{align}\det G_\lambda(X) = \det(\diag(\lambda), X).\label{eq:gadget_determinant} \end{align}
\end{lemma}
We delay the proof of Lemma \ref{lem:reduction_gadget} to Section \ref{subsec:reduction_gadget}, before which we show how to use Lemma \ref{lem:reduction_gadget} to construct and prove all the properties of the reduction.
\begin{definition} Think of $\mathbf{p}$ and $\mathbf{q}$ as partitions of $N$. Define 
$$\trun_{P,Q} T := G_{\mathbf{q}} \circ T \circ G_{\mathbf{p}}^*.$$
\end{definition}

\subsection{Scalability Under the Reduction}
We now show Item \ref{stoc_reds:item:scaling} of Theorem \ref{thm:stoc_reds} with $\GL_N(\CC) \times \GL_N(\CC)$ replaced by a smaller group $G\times H$; afterwards we will extend the proof using \cite{Gu04}.\\
 \indent Here is a useful observation that simplifies our techniques; the proof is an easy change of variables argument.
\begin{proposition}\label{prp:in_out} Suppose $P$ and $Q$ are nonsingular. Then
$T$ is $\GL(E_\bullet)\times \GL(F_\bullet)$-scalable to $(P \to I_m, Q \to I_m)$ if and only if 
$T$ is $\GL(E_\bullet)\times \GL(F_\bullet)$-scalable to $(I_m \to Q, I_n \to P)$.
\end{proposition}
The above proposition shows for $P,Q$ nonsingular, it is equivalent to characterize approximate $(P \to I_m, Q \to I_m)$-scalability.
\begin{lemma}\label{lem:reduction_smallergroup} Let $ G = G_{\mathbf{q}}\GL(E_\bullet)$, and $H =  G_{\mathbf{p}}\GL(F_\bullet)$. Then
$\trun_{P,Q} T$ is approximately
 $G \times H$-scalable to $(I_N \to I_N, I_N \to I_N)$ if and only if $T$ is approximately
  $\GL(E_\bullet)\times \GL(F_\bullet)$-scalable to $(P \to I_m, Q \to I_m)$.
  \end{lemma}
 
  \begin{proof}
\indent Assuming Lemma \ref{lem:reduction_gadget}, we can immediately see that\\ $\trun_{P,Q}T(I_N) = G_{\mathbf{q}} ( T(P))$, which is equal to $I_N$ if and only if $T(P) = I_m$. Since $(\trun_{P,Q} T)^* = \trun_{Q, P} (T^*)$, by symmetry we have
\begin{align}\trun_{P,Q} T \textrm{ is d. s. } \iff T \textrm{ maps } (P \to I_m, Q \to I_m).\label{eq:reduction_marginals}\end{align}
If $G'$ and $H'$ are groups and $S$ is a completely positive map, let $S_{G',H'} = \{S_{g,h}: g \in G', h \in H'\}.$ Then approximate $G'\times H'$-scalability of $S$ to $(A \to B, C\to D)$ is equivalent to the statement that $\overline{S_{G,H}}$ contains an element that maps $(A \to B, C \to D)$. The closure here can be taken in, say, the operator norm.\\
\indent In the next proposition we'll see that Eq. \ref{eq:gadget_hom} shows that scaling $T$ by $(g,h)$ corresponds to scaling $\trun_{P,Q} T$ by $(G_{\mathbf{q}} (g), G_{\mathbf{p}} (h))$. 
\begin{proposition}\label{prp:reduction_scaling}
Let $(g_0,h_0) \in \GL(E_\bullet)\times \GL(F_\bullet)$. Then 
$$ \trun_{P,Q} T_{g_0, h_0} = (\trun_{P,Q} T)_{G_{\mathbf{q}} (g_0), G_{\mathbf{p}} (h_0)}.$$
\end{proposition}
\begin{proof}
By definition, 
$$(\trun_{P,Q} T_{g_0,h_0})(X)
  = G_\mathbf{q}( g_0^\dagger T ( h_0 G_\mathbf{p}^*(X ) h_0^\dagger) g_0).$$
To complete the proof, apply Eq. \ref{eq:gadget_hom} and the equivalent dual version $(G_\mathbf{p}^*)_{h_0, I_N} = (G_\mathbf{p})_{I_N, h_0}^* = (G_\mathbf{p})^*_{G_\mathbf{p}(h_0), I_l} = (G_\mathbf{p}^*)_{I_l, G_\mathbf{p}(h_0)}.$
\end{proof}
We now finish the proof of Lemma \ref{lem:reduction_smallergroup}. By Proposition \ref{prp:reduction_scaling},
\begin{align*}
\overline{(\trun_{P,Q} T)_{G, H}} = \overline{\trun_{P,Q} T_{\GL(E_\bullet),\GL(F_\bullet)}}.
\end{align*}
$T \mapsto \trun_{P,Q} T$ is an injective, linear map, so 
$$\overline{\trun_{P,Q} T_{\GL(E_\bullet), \GL(F_\bullet)}} = \trun_{P,Q} \overline{T_{\GL(E_\bullet),\GL(F_\bullet)}}.$$
The above chain of equalities and Eq. \ref{eq:reduction_marginals} imply $\overline{(\trun_{P,Q} T)_{G, H}}$ contains a doubly stochastic element if and only if $\overline{T_{\GL(E_\bullet),\GL(F_\bullet)}}$ contains an element that maps $(P \to I_m, Q \to I_m)$, so we are done. 
\end{proof}

\begin{proof}[Proof of Item \ref{stoc_reds:item:scaling} of Theorem \ref{thm:stoc_reds}]
Lemma \ref{lem:reduction_smallergroup} and Proposition \ref{prp:in_out} imply that it is enough to show that $\trun_{P,Q} T$ is approximately
 $G \times H$-scalable to $(I_N \to I_N, I_N \to I_N)$ if and only if $\trun_{P,Q} T$ is approximately
 $\GL_N(\CC) \times \GL_N(\CC)$-scalable to $(I_N \to I_N, I_N \to I_N)$. The ``only if" implication is immediate, but the ``if" direction is subtler.\\
 \indent The key is that if $\trun_{P,Q} T$ is approximately
 $\GL_N(\CC) \times \GL_N(\CC)$-scalable to $(I_N \to I_N, I_N \to I_N)$, then by Theorem 4.6 in \cite{Gu04}, operator Sinkhorn iteration converges. That is, if we set $T_{0} = \trun_{P,Q} T$ and for $t \geq 1$ set 
 \begin{align*}T_{2t-1} = (T_{2t - 2})_{g(t),I_N} &\textrm{ and } T_{2t} = (T_{2t-1})_{I_N, h(t)}\\ \textrm{where }g(t)^\dagger T_{2t-2}(I_N) g(t) = I_N& \textrm{ and }h(t)^\dagger T_{2t-1}^*(I_N) h(t) = I_N,\end{align*}
then $T_t(I_N)$ and $T_t^*(I_N)$ converge to $I_N$.\\
 \indent Luckily, we may take $g(t) \in G$ and $h(t) \in H$! We only show this for $g$; the proof is similar for $h$. Inductively, suppose we have taken $g(s) \in G$ and $h(s) \in H$ for $s < t$. Let $g = g(1)\dots g(s)$ and $h = h(1)\dots h(s)$. Eq.\ref{eq:gadget_hom} implies in particular that $G_{\mathbf{p}}: \GL(F_\bullet) \to \GL_N(\CC)$ and $G_{\mathbf{q}}: \GL(E_\bullet) \to \GL_N(\CC)$ are group homomorphisms, and so $g = G_{\mathbf{q}}(g_0)$ and $h = G_{\mathbf{p}}(h_0)$ for some $(g_0,h_0) \in \GL(E_\bullet)\times \GL(F_\bullet)$. By Proposition \ref{prp:reduction_scaling}, 
 $$T_{2t-2} =  (\trun_{P,Q} T)_{g,h} = (\trun_{P,Q} T_{g_0,h_0}).$$
 Choose $g_1 \in \GL(E_\bullet)$ such that $g_1^\dagger T_{g_0,h_0}(P) g_1 =  T_{g_0g_1,h_0}(P) = I_m$ and set $g(t) = G_\mathbf{q}(g_1)$. This is possible by the existence of the Cholesky decomposition and the fact that the scaling procedure converges (in particular, $T_{g_0,h_0}(P) = T_{2t-2}(I_N)$ won't be singular). By the above identity, 
 $$g(t)^\dagger T_{2t-2}(X) g(t)^\dagger = (\trun_{P,Q} T_{g_0g_1,h_0})(P) = I_N.$$
 Thus, if $\trun_{P,Q}T$ is approximately $\GL_N(\CC) \times \GL_N(\CC)$-scalable to $(I_N \to I_N, I_N \to I_N)$, then it is approximately $G\times H$-scalable to $(I_N \to I_N, I_N \to I_N).$
\end{proof}
\subsection{Capacity Under the Reduction}
We still need to compute the capacity of $\trun_{P,Q} T$. We first show that the infimum in $\capacity \trun_{P,Q} T$ can be taken over the smaller group $H$, where $H$ is as in \ref{lem:reduction_smallergroup}, without changing the value.
\begin{lemma}\label{lem:capacity_prime}
Define 
$$\capacity' \trun_{P,Q} T = \inf_{h \in H}  \frac{\det (\trun_{P,Q} T) (hh^\dagger)}{\det hh^\dagger}.$$
Then
$$\capacity' \trun_{P,Q} T  = \capacity \trun_{P,Q} T.$$

\end{lemma}
\begin{proof}
We will show $\capacity' \trun_{P,Q} T  = \capacity \trun_{P,Q} T$ using the proof of Theorem \ref{thm:gurvits} in \cite{Gu04}. Clearly $\capacity' \trun_{P,Q} T  \geq \capacity \trun_{P,Q} T$. We need to show that $\capacity' \trun_{P,Q} T \leq \capacity \trun_{P,Q} T$. It is enough to prove this when $\capacity' \trun_{P,Q} T > 0$. Suppose this is true, and let $T_0, T_1, \dots$ be the operators resulting from operator Sinkhorn iteration on $\trun_{P,Q} T$. As shown in the Proof of Item \ref{stoc_reds:item:scaling} of Theorem \ref{thm:stoc_reds}, the scalings can be taken in the smaller groups $G$ and $H$. In \cite{Gu04} it is shown that at every step, $\capacity T_t$ increases by a factor that is at least $f(\epsilon)>1$ if the distance from doubly stochastic is at least $\epsilon$. Critically, the same argument shows $\capacity' T_t$ changes by the exact same factor in each step! Since $\capacity' T_t \leq 1$ for $t \geq 1$, this shows that if $\capacity'\trun_{P,Q} T > 0$, then $T_t(I_N)$, $T^*_t(I_N)$ converge to $I_N$. Since $\capacity T_t$ tends to one as $T_t(I), T_t^*(I)$ tend to the identity (\cite{GGOW16}, Lemma 2.27), $\capacity' T_t$ must also tend to one. Since $\capacity' T_t$ and $\capacity T_t$ changed by the same factor at each step, they must have been equal to begin with.\end{proof}
The remainder of the proof is purely computational.
\begin{proof}[Proof of Item \ref{stoc_reds:item:capacity} of Theorem \ref{thm:stoc_reds}] We just need to show that $\capacity' \trun_{P,Q} T = \capacity(T, P, Q)$. Define $h = G_\mathbf{p} (h_0)$ for $h_0 \in \GL(F_\bullet)$. By Proposition \ref{prp:reduction_scaling}, Eqs. \ref{eq:gadget_marg},  \ref{eq:gadget_hom}, and Eq. \ref{eq:gadget_determinant},
\begin{align*}\capacity' \trun_{P,Q} T 
&=\inf_{h \in H}  \frac{\det (\trun_{P,Q} T)_{I_N, h}(I_N)}{\det h^\dagger h}\\
&=\inf_{h_0 \in \GL(F_\bullet)}  \frac{\det  T_{I_m, h_0}(P)}{\det G_\lambda(h_0 h_0^\dagger)}\\
& = \capacity(T, P, Q).
\end{align*}
\end{proof}
\subsection{Proof of Lemma \ref{lem:reduction_gadget}}\label{subsec:reduction_gadget}
Here we construct the promised $G_\lambda$ in Lemma \ref{lem:reduction_gadget}. A partition $\lambda$ is often depicted by a Young diagram, a left-justified collection of boxes with $\lambda_i$ boxes in the $i^{th}$ row from the top. The conjugate partition $\lambda'$ of a partition $\lambda$ is the partition obtained by transposing the Young diagram. For example, if $\lambda = (3,1)$, then
 $$ \lambda = \yng(3,1) \textrm{ and } \lambda' = \yng(2,1,1).$$
\begin{definition} Let $\lambda$ be a partition of $l$ with $k$ parts, and $\lambda'$ its conjugate partition. Define $ G_\lambda:\Mat_{k\times k} \CC \to \Mat_{l \times l} \CC$ by 
\begin{align} G_\lambda: X \mapsto \bigoplus_{i = 1}^{\lambda_1} \eta_{\lambda'_i} X\eta_{\lambda'_i}^\dagger,\label{eq:gadget_definition}\end{align}
where again $\eta_j:\CC^k \to \CC^j$ denotes the projection to the first $j$ coordinates.
\end{definition}
\begin{example} If $\lambda = (2,2,1)$, then 
$$G_\lambda \left(\left[ 
\begin{array} {cc|c} a & b & c \\
d & e & f  \\
\hline
g & h & i   \\
\end{array}
 \right]\right) = \left[ 
\begin{array} {ccc|cc} 
a & b & c & 0 & 0 \\
d & e & f & 0 & 0 \\
g & h & i  & 0 & 0 \\
\hline
0 & 0 & 0 & a & b \\
0 & 0 & 0 & d & e \\

\end{array}
 \right]$$
\end{example}
\begin{proof}[Proof of Lemma \ref{lem:reduction_gadget}]
The proofs are quite simple. We first show Eq. \ref{eq:gadget_marg}. It is clear that $G_\lambda(I_k) = I_l$. Next, observe that
$$G_\lambda^*(I_l) = \sum_{i = 1}^{\lambda_1} \eta_{\lambda'_i}^\dagger \eta_{\lambda'_i} = \diag(\lambda).$$
The second equality holds because $ \eta_{\lambda'_i}^\dagger \eta_{\lambda'_i}$ is a matrix with $\lambda'_i$ ones descending down the diagonal - the ones along the diagonal correspond to boxes in the columns of the Young diagram.\\
\indent Eq. \ref{eq:gadget_hom} follows by applying the following handy fact about upper triangular matrices to each direct summand in \ref{eq:gadget_definition}.
 \begin{proposition}\label{prp:upper_triangular} Let $h$ be upper triangular. Then 
 $$h \eta_j^\dagger = \eta_j^\dagger \eta_j h \eta_j^\dagger.$$
 \end{proposition}
 \begin{proof} One can draw a picture. Alternatively, $\eta_j^\dagger$ is an embedding of $\CC^j$ into the span $E_j$ of the first $j$ coordinate vectors; $h$ fixes $E_j$. The map $\eta_j^\dagger \eta_j$ is a projection to $E_j$, so acts as the identity on $E_j$.
 \end{proof}
 Finally, Eq. \ref{eq:gadget_determinant} is just counting - the number of times $\eta_j X \eta_j^\dagger$ appears as a direct summand in \ref{eq:gadget_definition} is $\Delta \lambda_i = \lambda_i - \lambda_{i+1}$. 
\end{proof}
\begin{proposition}[Kraus operators of $\trun_{P,Q} T$.]\label{prp:reduction_kraus}
If $T$ had Kraus operators $(A_i:i \in [r])$, then in an appropriate basis $\trun_{P,Q} T$ will have Kraus operators indexed by $i \in [r]$, $j \in [q_1]$ and $k \in [p_1]$ given by 
\begin{align*}
\left[\begin{array}{ccccc} 0_{q'_1, p'_1} & \hdots &  0_{q'_1, p'_k} & \hdots& 0_{q'_1, p'_{p_1}} \\
\vdots &\ddots&\vdots &\ddots&\vdots\\
0_{q'_j, p'_1}  & \hdots&  \eta_{q'_j} A_i \eta_{p'_k}^\dagger &\hdots &0_{q'_j, p'_{p_1}} \\
 \vdots &\ddots&\vdots &\ddots&\vdots\\
0_{q'_{q_1}, p'_1}  & \hdots & 0_{q'_{q_1}, p'_k} & \hdots & 0_{q'_{q_1}, p'_{p_1}}
\end{array}\right]. 
\end{align*}
That is, they are block matrices with $\eta_{q'_j} A_i \eta_{p'_k}^\dagger$ in the $j,k$ spot and zeroes elsewhere. 
\end{proposition}


\section{Triangular Scalings}\label{sec:triangular_scalings}
This section contains the proof of Theorem \ref{thm:apx_delta_scalable}, as well as some bounds on the running time of our algorithms.

\subsection{Upper Triangular Scaling Algorithm}\label{subsec:triangular_algorithm}
Here we show the implication $\ref{delt_cap} \implies \ref{delt_scal}$ of Theorem \ref{thm:apx_delta_scalable}. By Proposition \ref{prp:in_out}, it is enough to show $\capacity(T, P, Q) > 0$ implies approximate $(P \to I_m, Q \to I_n)$-scalability. The ideas here follow the ideas in \cite{Gu04} very closely. We use Jensen's inequality to show capacity increases by a function of $\epsilon$ in each step of Algorithm \ref{alg:alg_g_tri} unless $T$ is already an $\epsilon$-$(P \to I_m, Q \to I_n)$-scaling, and we show it is always bounded by one. Thus, if the capacity is nonzero to begin with, Sinkhorn scaling eventually results in an $\epsilon$-$(P \to I_m, Q \to I_n)$-scaling. 

\subsubsection*{Reducing to the nonsingular case}
 It is much more convenient to work with nonsingular $P$ and $Q$ analyze Algorithm \ref{alg:alg_g_tri}. If $P$ and $Q$ are singular, we can simply project the operator to $\supp P$ and $\supp Q$ - namely, the span of the positive eigenspaces.
\begin{definition} Define $\eta := \eta_{\rank P}$; $\eta$ is a partial isometry such that $\eta^\dagger \eta$ is an orthogonal projection to $\supp P$. Let $\nu$ denote the analogous isometry for $Q$. Define 
$$\underline{T}: \Mat_{\rank P \times \rank P} \CC \to \Mat_{\rank Q \times \rank Q} \CC$$
by 
$$\underline{T}: X \mapsto \nu T(\eta^\dagger X \eta )\nu^\dagger,$$
and define $\underline{P} := \eta P \eta^\dagger \succ 0$ and $\underline{Q} = \nu Q \nu^\dagger \succ 0$. Define $\underline{G} = \nu \GL(E_\bullet) \nu^\dagger $ and $\underline{H} = \eta \GL(F_\bullet) \eta^\dagger;$ these groups are actually the upper triangular invertible matrices in $\GL_{\rank Q}(\CC)$, $\GL_{\rank P}(\CC)$, respectively.
\end{definition}
Here we reduce to the nonsingular case - Lemma \ref{inv_red_lite} below implies it is enough to prove $\ref{delt_cap} \implies \ref{delt_scal}$ of Theorem \ref{thm:apx_delta_scalable} when $P$ and $Q$ are nonsingular.
\begin{lemma}\label{inv_red_lite}$ $
\begin{enumerate}
\item\label{cap_lower} $\capacity (\underline{T}, \underline{P}, \underline{Q})  \geq \capacity(T, P, Q)$.
\item \label{scal_lifts}If $T$ has an $\epsilon$-$(I_{\rank P} \to \underline{Q}, I_{\rank Q} \to \underline{P})$-scaling by $\underline{G}\times\underline{H}$, then $T$ has a\\ $2\epsilon$-$(I_m \to Q, I_n \to P)$-scaling by  $\GL(E_\bullet) \times \GL(F_\bullet)$.
\end{enumerate}
\end{lemma}
The proof is quite simple and can be found in Appendix \ref{app:triangular_scalings}.

\subsubsection*{Distance to doubly stochastic}
\indent We will need another notion of how far $T$ is from being a $(P \to I_m, Q \to I_n)$-scaling. The quantity $\ds_{P,Q} T$ is a natural distance measure because it agrees with $\ds \trun_{P,Q} T$ when $P$ and $Q$ have integral spectra.
\begin{definition}\label{dfn:ds_def}
Let $\ds_{P,Q} T = $
\begin{align*}
 \sum_{i =1}^n \Delta p_i \left\|\eta_i ( T^* ( Q ) - I_n) \eta_i^\dagger \right\|^2 + \sum_{j =1}^m \Delta q_j \left\| \eta_j (  T( P )  - I_m) \eta_j^\dagger \right\|^2.
\end{align*}
In particular, $\ds_{P,Q}T \geq p_n \|T^*(Q) - I_m\|^2 + q_m \|T(P) - I_n\|^2$, so if $P$ and $Q$ are invertible and $\ds_{P,Q} T_{g,h} < \epsilon^2 \min \{p_n, q_m\}$ then $T_{g,h}$ is an $\epsilon$-$(P \to I_m, Q \to I_n)$-scaling of $T$.
\end{definition}

\begin{Algorithm}[th!]
\begin{flushleft}
\textbf{Input:} $T$, $P, Q$ as in Section \ref{subsec:prelim}. In addition $P$ and $Q$ are nonsingular.\\
 \vspace{.25cm}
\textbf{Output:} A pair $(g,h) \in \GL(E_\bullet) \times \GL(F_\bullet)$ such that $T_{g,h}$ is an $\epsilon$-$(P \to I_m, Q \to I_n)$ scaling of $T$, or ERROR.\\
 \vspace{.25cm}
\textbf{Algorithm:}
\begin{enumerate}
\item 
Set $g_0 = I_m$, $h_0 = I_n$.
\item For $j \in [TIME]$:
\begin{enumerate}
\item 
\begin{description}
\item[If $j$ is odd:] Find $g\in \GL({E_\bullet})$ such that $ g^\dagger T (h_{j-1}Ph_{j-1}^\dagger)g  = I$. Set $g_j = g$ and $h_j = h_{j-1}$.
\item[If $j$ is even:] Find $h \in \GL({F_\bullet})$ such that $ h^\dagger T^* (g_{j-1}Qg_{j-1}^\dagger)h  = I$. Set $g_{j} = g_{j-1}$ and  $h_j = h$.
\end{description}
If this was not possible, \textbf{return} ERROR.
\item If $ \ds T_{g_j,h_j} \leq \epsilon$,
 \textbf{return} $(g_j, h_j)$. 
\end{enumerate}
\item \textbf{Return} ERROR.
\end{enumerate}
\caption{Upper triangular scaling algorithm.}\label{alg:alg_g_tri}
\end{flushleft}
\end{Algorithm}
\begin{definition}
It will be convenient to define $T_{j} := T_{g_j, h_j}$ where $g_j, h_j$ are as in Algorithm \ref{alg:alg_g_tri}.
\end{definition}
\begin{theorem}\label{thm:alg_tri_term_stoc}
 If $\capacity(T, P, Q) > 0,$ and
$$TIME = \frac{ - 7\log \capacity(T_1, P, Q)}{\min\{\epsilon, p_n\} + \min\{\epsilon, q_m\}}, $$
then Algorithm \ref{alg:alg_g_tri} does not output ERROR.
\end{theorem}

Here are the results needed to prove Theorem \ref{thm:alg_tri_term_stoc} as per the plan outlined at the beginning of Section \ref{subsec:triangular_algorithm}; we hint here at the proofs and the full versions can be found in Appendix \ref{app:triangular_scalings}. 
\begin{lemma}[nonsingularity]\label{lem:dual_inv}
Suppose $P$ and $Q$ are invertible and $\capacity(T, P, Q) > 0$. Then $T$ and $T^*$ both map positive-definite operators to positive-definite operators.
\end{lemma}
 The proof of Lemma \ref{lem:cap_evol} is an easy change-of-variables argument.\\

\begin{lemma}[capacity update]\label{lem:cap_evol}
If $h \in  \GL(F_\bullet)$ and $g \in \GL(E_\bullet)$, 
then 
$$\capacity(T_{g,h}, P,Q) =   \det(Q, g^\dagger g) \det(P, h^\dagger h) \capacity(T, P,Q).$$
\end{lemma}
The next lemma follows from Jensen's inequality applied to the eigenvalues of $\eta_j T^*(Q)\eta_j^\dagger$. 
\begin{lemma}[substantial progress]\label{lem:cap_incr}
Suppose $T(P) = I$, $\Tr P = \Tr Q = 1$, $\ds_{P,Q}T \geq \epsilon$, and $h\in H_{F_\circ(P)}$ such that $h^\dagger T^*(Q) h = I.$
Then $$\det(P, h^\dagger h) \geq e^{ .3 \min\{\epsilon, p_n\}}.$$
\end{lemma}
\begin{lemma}[capacity upper bound]\label{lem:cap_ub}
Suppose $T(P) = I_m$ or $T^*(Q) = I_n$. Then $\capacity( T, P, Q) \leq 1.$
\end{lemma}

We now assemble the lemmas:
\begin{proof}[Proof of Theorem \ref{thm:alg_tri_term_stoc}]The existence of the Cholesky decomposition and Lemma \ref{lem:dual_inv} imply each step is possible; e.g. $gg^\dagger$ will be the Cholesky decomposition of $T(h_{j - 1} P h_{j - 1})^{-1}$. Since $\capacity(T, P, Q) > 0$, Lemma \ref{lem:cap_evol} implies $\capacity( T_1, P, Q)>0$. Suppose $j \geq 2$ and $j$ even. Provided $\ds_{P,Q} T_j \geq \epsilon$, Lemmas \ref{lem:cap_evol} and \ref{lem:cap_incr} imply
$$\capacity( T_j, P, Q) \geq e^{.3\min\{p_n,\epsilon\}} \capacity( T_j, P, Q).$$
 If $j$ is odd and $j \geq 3$, then Lemma  \ref{lem:cap_evol} and Lemma \ref{lem:cap_incr} applied to $T^*$ with the roles of $P$ and $Q$ reversed and the roles of $G$ and $H$ reversed implies 
$\capacity( T_j, P, Q) \geq e^{.3\min\{q_m,\epsilon\}} \capacity( T_j, P, Q).$
By the very easy Lemma \ref{lem:cap_ub},  $\capacity(T_j, P, Q) \leq 1$ for $j \geq 1$. A bit of algebra shows Algorithm \ref{alg:alg_g_tri} terminates in at most $TIME$ iterations.
\end{proof}

\subsection{Proof of Theorem \ref{thm:apx_delta_scalable}}\label{subsec:apx_delta_scalable_proof}
Here we finish the proof of Theorem \ref{thm:apx_delta_scalable} by showing $\ref{delt_rnd}\implies \ref{delt_cap}$. We prove the stronger statement that approximate scalability implies approximate rank-nondecreasingness.
\begin{proposition}[scalability implies rank-nondecreasingness]\label{prp:apx_rnd_stoc}
Suppose there exists an $\epsilon$-$(I_m \to  Q, I_n \to P)$-scaling of $T$ by $\GL(E_\bullet)\times\GL(F_\bullet)$. Then $|\Tr Q - \Tr P| \leq (\sqrt{n} + \sqrt{m})\epsilon$ and 
\begin{equation*} \sum_{i=1}^n \Delta q_i \dim E_i \cap L + \sum_{j =1}^n \Delta p_i \dim F_j \cap R  \leq \Tr P + (\sqrt{n} + \sqrt{m})\epsilon 
  \end{equation*}
for every $T$-independent pair $(L,R)$. 
\end{proposition}
\begin{proof}

Suppose there exists an $\epsilon$-$(I_V \to  Q, I_W \to P)$-scaling $T_{g,h}$ of $T$ by $\GL(E_\bullet)\times \GL(F_\bullet)$. Firstly, 
\begin{align*}
|\Tr Q - \Tr P|\\
\leq |\Tr (Q - T_{g,h}(I))| + |\Tr (P - T_{h,g}^*(I))|\\
\leq (\sqrt{n} +\sqrt{m})\epsilon.
\end{align*}
Next, consider any $T$-independent pair $(L,R)$. Observe that $(\underline{L},\underline{R}) =(g^{-1} L, h^{-1} R)$ is a $T_{g,h}$-independent pair and $\dim \underline L \cap E_i = \dim L \cap E_i$ for all $i \in [m]$ and $\dim \underline R \cap F_i = \dim R \cap F_i$ for all $i \in [n]$ by membership of $(g,h) \in \GL(E_\bullet)\times \GL(F_\bullet)$.

We use a standard fact, which is essentially an inequality for the Rayleigh trace. See \cite{F00}.
\begin{fact} Suppose $ \textbf a$ is a nonincreasing sequence of positive real numbers. Set $A = \diag( \textbf a)$. Let $E_\bullet$ be the standard flag of $\CC^d$, and $\pi_L$ the orthogonal projection to the subspace $L \subset \CC^d$. Then 
$$ \Tr A \pi_L \geq \sum_{i = 1}^d \Delta a_i \dim L \cap E_i $$ with equality if and only if $\pi_L$ commutes with $A$.
\end{fact}
Since $(\underline{L},\underline{R})$ is $T$-independent, $T_{g,h}^*(\pi_{\underline{L}})\pi_{\underline{R}} = 0$. Thus \begin{align*}
\sum_{i = 1}^m \Delta q_i \dim L \cap E_i \leq \Tr Q \pi_{\underline{L}} \nonumber &\leq \Tr T_{g,h}(I) \pi_{\underline{L}} + \epsilon \sqrt{m}  \nonumber\\
&= \Tr T^*_{h,g}( \pi_{\underline{L}}) +  \epsilon  \sqrt{m}\nonumber\\
&= \Tr T^*_{h,g}( \pi_{\underline{L}})(I - \pi_{\underline{R}}) +   \epsilon\sqrt{m} \nonumber\\
&\leq \Tr T^*_{h,g}(I)(I - \pi_{\underline{R}}) + \epsilon \sqrt{m}  \label{equality?}\\
&\leq \Tr P(I - \pi_{\underline{R}}) +  (\sqrt{n} + \sqrt{m}) \epsilon \nonumber\\
&= \Tr P - \sum_{j =1}^n \Delta p_i \dim F_j \cap R + (\sqrt{n} + \sqrt{m}) \epsilon.\nonumber
\end{align*} \end{proof}

\begin{corollary}\label{cor:sure_rnd}
If $\epsilon$ is smaller than the minimum nonzero number among 
\begin{equation}\left\{\frac{1}{\sqrt{m} + \sqrt{n}} \left| \Tr P - \sum_{i \in I} q_i - \sum_{j \in J} p_j.\right|: I \subset [m], J \subset [n]\right\}\label{ds_to_rnd},\end{equation}
and there exists an $\epsilon$-$(I_n \to  Q, I_m \to P)$-scaling of $T$ by $\GL(E_\bullet)\times \GL(F_\bullet),$ then $T$ is $(P,Q)$-rank-nondecreasing. 
\end{corollary}

 We first show that the set of $\mathbf{p}, \mathbf{q}$ such that $\capacity(T, P, Q)$ is nonzero is convex. This immediately follows from the next proposition:

\begin{proposition}[$\log$-concavity]\label{prp:log_concave}
$e^{H(\mathbf{p})} \capacity(T, P, Q)$ is $\log$-concave in $\mathbf{p}$ and $\mathbf{q}$.
\end{proposition}
\begin{proof} First, one can make a change of variables argument, which we omit, to see that 
$$\capacity(T, P,Q) = e^{-H(\mathbf{p})} \inf_{h \in \GL(F_\bullet)} \frac{\det(Q, T(hh^\dagger))}{\det(P, h^\dagger h )}.$$
However, the left-hand side is manifestly $\log$-convex in $\mathbf{p}$ and $\mathbf{q}$! To see why, let's expand. 
\begin{align*}
\log \inf_{h \in \GL(F_\bullet)} \frac{\det(Q, T(hh^\dagger))}{\det(P, h^\dagger h )} =\inf_{h \in  \GL(F_\bullet)} \log \frac{\det(Q, T(hh^\dagger))}{\det(P, h^\dagger h )}\\
= \inf_{h \in  \GL(F_\bullet)} \sum_{i=1}^m \Delta q_i \log \det \eta_i T\left(h h^\dagger \right)\eta_i^\dagger - \sum_{j = 1}^n \Delta p_j \log \det \eta_j h^\dagger h \eta_j^\dagger 
\end{align*}
with the convention $0 \log 0 = 0$. Is of the form $f:\RR^{m + n}_+ \times U \to \RR \cup \{-\infty\}$
defined by
$$f:(x,u) \mapsto  \langle x , g(u) \rangle $$ where $U$ is some set; in this case $U =  \GL(F_\bullet)$. It is easy to check that $f$ is always concave when $x, g(u)$ remain finite; with slightly more care one can check it when $g(u)_i \in \RR \cup \{-\infty\}$ with the convention $x_i g(u)_i = 0$ when $x_i = 0$ and $g(u)_i = - \infty$.
\end{proof}
We are ready to prove the theorem. 
\begin{proposition}
Theorem \ref{thm:apx_delta_scalable} holds; further, the set of $(\mathbf{p}, \mathbf{q})$ each with unit sum such that any of the three conditions hold is a convex polytope $\cK(T, E_\bullet, F_\bullet)$ with rational vertices.
\end{proposition}
\begin{proof}
Define
$$\cC(T, E_\bullet, F_\bullet) \subset \cS(T, E_\bullet, F_\bullet) \subset \cK(T, E_\bullet, F_\bullet),$$ the set of pairs $(\mathbf{p}, \mathbf{q})$ each summing to one such that, respectively, $\capacity(T, P, Q) > 0$, $T$ is approximately $(I_n \to P, I_m \to Q)$-scalable, and $T$ is $(P,Q)$-rank-nondecreasing. The left inclusion is Theorem \ref{thm:alg_tri_term_stoc}, along with Lemma \ref{inv_red_lite} and Proposition \ref{prp:in_out}, and the right inclusion is Proposition \ref{prp:apx_rnd_stoc}.\\
\indent The reduction characterizes exactly the intersection of each of these sets with $\QQ^{n + m}$ (we may always scale so that rational $\mathbf{p}, \mathbf{q}$ become integral without changing scalability, rank-nondecreasingness, or nonvanishing of capacity). By Theorem \ref{thm:stoc_reds},
$$ \cC(T, E_\bullet, F_\bullet)\cap \QQ^{n + m} = \cK(T, E_\bullet, F_\bullet)\cap \QQ^{n + m}.$$
Since $\cK(T, E_\bullet, F_\bullet)$ is a convex polytope with rational vertices, $\cK(T, E_\bullet, F_\bullet)\cap \QQ^{n + m}$ contains the vertices of $\cK(T, E_\bullet, F_\bullet)$! However, $\cC(T, E_\bullet, F_\bullet)$ is convex by Proposition \ref{prp:log_concave}, so it contains $\cK(T, E_\bullet, F_\bullet)$. The three sets must be the same; this completes the proof.
\end{proof}

\subsection{Running Time of Algorithm \ref{alg:alg_g_tri}}
In order to use the guarantees from the previous subsection to bound the running time of Algorithm \ref{alg:alg_g_tri}, we must bound the capacity below. For this we will need to use a nontrivial lower bound on $\capacity T$ from \cite{GGOW16}.  
\begin{theorem}[Garg et. al. \cite{GGOW16}]\label{thm:cap_lower_bd}
If $T:\Mat_{N\times N}\CC \to \Mat_{N\times N}\CC$ is a rank-nondecreasing completely positive operator with Kraus operators $A_1 \dots A_R$ with Gaussian integer entries, then $$\capacity T \geq e^{-N\log (R N^4)}.$$
\end{theorem}

This is implicit in the proof of Theorem 2.21 in \cite{GGOW16}, which gives the bound $\exp(- O(N^2 \log (R N^4))$. One of the bounds used there has since been improved; we discuss this degree bound, since we will use it later.

\begin{definition}\label{dfn:derksen_polies} Let $T:\Mat_{N\times N}\CC \to \Mat_{N\times N}\CC$ be a completely positive map with Kraus operators $A_1 \dots A_R$. The quantity $\sigma(N,R)$ is defined to be the minimal $d$ such that $T$ is rank-nondecreasing if and only if the polynomial $p:\Mat_{d\times d}(\CC)^R \to \CC$ given by $$p(B_1, \dots, B_R) =\det\left(\sum_{i = 1}^{R} A_i \otimes B_{i}\right)$$
is not identically zero.
\end{definition}
It is interesting that $\sigma(N,R)$ even exists. The bound $\sigma(N,R) \leq (N+1)!$ was used in \cite{GGOW16}, but a better bound appeared afterwards.
\begin{theorem}[Derksen, Makam \cite{DM17}]\label{thm:derksen}
$\sigma(N,R) \leq N-1$. 
\end{theorem}
We now prove our lower bound. 
\begin{theorem}\label{thm:stoc_lower_bd}
Suppose $T, P, Q$ are of bit-complexity at most $b$ as per Definition \ref{dfn:bit_complexity}. If $\capacity(T, P, Q)> 0$, then
$$\capacity(T, P, Q) \geq \exp( - 10 b).$$
\end{theorem}
\begin{proof} Choose an integer $N \leq 2^{b}$ such that $\mathbf{p}' = N \mathbf{p}$ and $\mathbf{q}' = N \mathbf{q}$ have integer entries and such that $N^2 T$ has Gaussian integer entries. First note that if $T'$ is the completely positive map obtained by scaling the Kraus operators of $T$ by $N$, then the Kraus operators of $T'$ have Gaussian integer entries and $\capacity(T', \mathbf p, \mathbf q) = N^2\capacity(T,\mathbf p,\mathbf q).$ By Proposition \ref{prp:reduction_kraus},  $\trun_{N P,N Q} T'$ has $N^2 p_1q_1r \leq N^2 r$ many $N \times N$ Kraus operators filled with Gaussian integer entries. Further, $\capacity(T, P, Q)>0$ then $\trun_{N P, N Q} T'$ is rank-nondecreasing. By Theorem \ref{thm:cap_lower_bd}, 
$$\capacity \trun_{N P,N Q} T' > e^{N \log (N^6r)}.$$
However, it's not hard to check that 
$$\capacity(T', P, Q) = \frac{1}{N}\capacity(T', NP, NQ)^{1/N}, $$
and we know $\capacity \trun_{N P,N Q} T' = \capacity(T', NP, NQ)$. Therefore 
$$ \capacity (T, P, Q) \geq \frac{1}{N^3} e^{ - \log (N^6r)}\geq e^{-9\log(Nr)}.$$
Using $N \leq 2^b$ and $b \geq \log_2 r$ completes the proof. \end{proof}

\indent We also need to ensure that the capacity does not decrease too much after the first step of Algorithm \ref{alg:alg_g_tri}. The proof is straightforward and is in Appendix \ref{app:missing_proofs}.
\begin{lemma}\label{lem:scaled_capacity_lower_bound} Let $T, P, Q$ have bit-complexity at most $b$ and\\ $\capacity(T, P, Q) > 0$. If $T_1$ is the operator obtained from the first step of Algorithm \ref{alg:alg_g_tri} applied to $T, P, Q$, then 
$$ \capacity(T_1, P, Q) \geq e^{- 14 b m}.$$ 
\end{lemma}
Now we can just plug the bound from Lemma \ref{lem:scaled_capacity_lower_bound} into Theorem \ref{thm:alg_tri_term_stoc}.
\begin{corollary}\label{cor:tri_quant_term}
Let $T, P, Q$ have bit-complexity at most $b$ and\\ $\capacity(T, P, Q) > 0$. Then Algorithm \ref{alg:alg_g_tri} does not output ERROR if
$$TIME  \geq \frac{100 bm }{\min\{\epsilon, p_n\} + \min\{\epsilon, q_m\}}.$$ 
\end{corollary}
The bound in Theorem \ref{thm:stoc_lower_bd} can be improved to be dependent on only the bit-complexity of $A_1, \dots, A_r$ by using Proposition \ref{prp:log_concave} - any $\log$-concave function on a convex polytope takes its minimum on a vertex. The extreme points of $\cK(T, E_\bullet, F_\bullet)$ are rational and have worst-case bit-complexity depending polynomially on $m$ and $n$; one runs the proof of Theorem \ref{thm:stoc_lower_bd} on the vertex where the minimum is attained. A further improvement can be made by assuming the number of Kraus operators is bounded by $nm$, which is without loss of generality, though we omit the proof.

\section{General linear scalings}\label{sec:general_scalings}
The proof of Theorem \ref{thm:gln_scalable} is quite simple once we have the following lemma, which essentially says that if there are $B$ and $C$ such that $T_{B, C}$ is be $(P,Q)$-rank-nondecreasing, then $T_{B,C}$ is $(P,Q)$-rank-nondecreasing for generic $B,C$.
\begin{definition}\label{dfn:affine_variety}
The set of common zeroes of a collection of polynomials in $\CC[x_1, \dots, x_n]$ is called an \emph{affine variety} in $\CC^d$. We say a property holds for \emph{generic} $x \in S$ if it holds for all $x$ in $S \setminus V$ for some fixed affine variety $V$ not containing $S$.
\end{definition}

\begin{lemma}\label{lem:exists_iff_generic_stoc} The set of pairs $(B^\dagger, C) \in \Mat_{m\times m}\CC \times \Mat_{n\times n}\CC$ such that $T_{B, C}$ is $(P,Q)$-rank-nondecreasing is the complement of an affine variety $V(T, P, Q)$. Further, if $P$ and $Q$ have integral spectra, then $V(T, P, Q)$ is generated by finitely many polynomials of degree at most $2(\Tr P)^2$.\end{lemma}

\subsection{Proof of Theorem \ref{thm:gln_scalable}}

We first show how to prove Theorem \ref{thm:gln_scalable}, then we give a hint towards proving Lemma \ref{lem:exists_iff_generic_stoc}.
\begin{proof}[Proof of Theorem \ref{thm:gln_scalable}]
We first prove \ref{gen_scal}$\implies $\ref{gen_p_q}. If $T$ is approximately $\GL_m(\CC)\times \GL_n(\CC)$-scalable to $(I_V \to  Q, I_W \to P)$, then by Corollary \ref{cor:sure_rnd}, there exists $(g,h) \in \GL_m(\CC)\times \GL_n(\CC)$ such that $T_{g,h}$ is $(P,Q)$-rank-nondecreasing. By Lemma \ref{lem:exists_iff_generic_stoc}, 
$$\{(g^\dagger ,h) \in \GL_m(\CC)\times \GL_n(\CC): T_{g,h} \textrm{ is }(P,Q)\textrm{-rank-nondecreasing}\}$$
is nonempty and the complement of an affine variety.

This shows $T_{g,h}$ is $(P,Q)$-rank-nondecreasing for generic $(g^\dagger, h) \in \GL_m(\CC)\times \GL_n(\CC)$. \ref{gen_p_q}$\implies$\ref{gen_cap} follows from Theorem \ref{thm:apx_delta_scalable}. Next we show \ref{gen_cap}$\implies$\ref{gen_scal}. Suppose $\capacity(T_{g,h}, P, Q) > 0$ for generic $(g^\dagger, h) \in \GL_m(\CC)\times \GL_n(\CC)$. In particular, there exists $(g,h) \in \GL_m(\CC)\times \GL_n(\CC)$ such that $\capacity(T_{g,h}, P, Q) > 0$. By Theorem \ref{thm:apx_delta_scalable}, $T_{g,h}$ is approximately $\GL(E_\bullet)\times \GL(F_\bullet)$-scalable to $(I_n \to  Q, I_m \to P)$, so $T$ is approximately $g\GL(E_\bullet)\times h\GL(F_\bullet)$-scalable to $(I_n \to  Q, I_m \to P)$. Because $\GL(E_\bullet)\times \GL(F_\bullet) \subset \GL_m(\CC)\times \GL_n(\CC)$, $T$ is approximately $\GL_m(\CC)\times \GL_n(\CC)$-scalable to $(I_n \to  Q, I_m \to P)$.
\end{proof}
We now prove Lemma \ref{lem:exists_iff_generic_stoc}.
\begin{proof}
If $\mathbf{p}$ and $\mathbf{q}$ are rational, the polynomials are those in Definition \ref{dfn:derksen_polies} computed from the Kraus operators of $\trun_{P,Q} T$ after scaling $\mathbf{p}$ and $\mathbf{q}$ to have integral spectra. Theorem \ref{thm:derksen} and Proposition \ref{prp:reduction_kraus} implies the degree bound. If $\mathbf{p}$ or $\mathbf{q}$ need not be rational, we must be more careful.\\
\indent First one shows that for any fixed pair $\mathbf{c} \in [l]^n, \mathbf{d} \in [k]^m$ of sequences of nonnegative integers, the set of tuples $(B^\dagger, C, L^\dagger, R)$ in
$$\Mat_{m\times m}\CC \times \Mat_{n\times n}\CC \times \Mat_{k\times m}\CC \times \Mat_{m \times l} \CC$$ such that $\row L, \row R$ is $T_{B,C}$-independent, and $L$ and $R$ are full-rank and satisfy 
\begin{align}\dim \row L \cap E_i \geq d_i \textrm{ and }\dim \row R \cap F_i \geq d_i \label{eq:double_schubert}
\end{align} is an \emph{constructible set} $S(\mathbf{c}, \mathbf{d})$, namely it is a union of sets of the form $V\setminus E$ where $E$ and $V$ are affine varieties. This is true because $T_{B,C}$ independence of $\row L, \row R$ is equivalent to the vanishing of the (polynomial) entries of $L^\dagger B^\dagger A_i C R$ for all $i \in [r]$, and the set of $L, R$ satisfying \ref{eq:double_schubert} is precisely a pair of matrices whose row spaces belong to certain Schubert varieties as in Remark \ref{rem:schubert}. Such sets of matrices are constructible because membership of $\row L$ in the Schubert variety corresponding to $\textbf d$ is equivalent to the nonvanishing of at least one $k \times k$ minor of $L$ and the vanishing of a certain subset (depending on $\textbf d$) of the $k\times k$ minors of $L$ \cite{Le10}.

\indent Next, consider the map $\pi:(B^\dagger, C, L^\dagger, R) \to (B^\dagger, C)$. Chevalley's theorem (see \cite{M99}) says $S'(\mathbf{c}, \mathbf{d}) :=\pi S(\mathbf{c}, \mathbf{d})$ is also constructible. However, it's also not hard to see that $S'(\mathbf{c}, \mathbf{d})$ is closed in the Euclidean topology, so it is in fact an affine variety (see \cite{MilneAG}). If we now take $D$ to be the set of $(\mathbf{c}, \mathbf{d})$ such that $\sum \Delta p_i c_i + \Delta q_i d_i > \Tr P$, then $V(T, P, Q) = \bigcup_{(\mathbf{c}, \mathbf{d}) \in D}S'(\mathbf{c}, \mathbf{d})$ is also an affine variety, and is precisely the set of $(B^\dagger, C)$ such that $T_{B,C}$ is \emph{not} $(P,Q)$-rank-nondecreasing. 
\end{proof}

The proof of Lemma \ref{lem:exists_iff_generic_stoc} also shows the following:
\begin{proposition}\label{prp:poly} The set
\begin{align*}\cK(T) = \{(\mathbf{p}, \mathbf{q}): \Tr P = \Tr Q = 1  \textrm{ and } T \textrm{ is approximately }\\
  (I_n \to Q, I_m \to P)\textrm{-scalable}\}
 \end{align*}
 is a convex polytope with rational vertices.
\end{proposition}
\begin{proof} 
Take $D'$ to be the set of $(\mathbf{c}, \mathbf{d})$ in the proof of Lemma \ref{lem:exists_iff_generic_stoc} such that $S'_{\mathbf{c}, \mathbf{d}}$ is all of $\Mat_{m\times m}\CC\times \Mat_{n\times n}\CC$. Since the union of the varieties $S'_{\mathbf{c}, \mathbf{d}}$ for $(\textbf c, \textbf d) \in ([l]^m \times [k]^n) \setminus D'$ is a proper affine variety, $D'$ are the constraints that generically arise for $\cK(T_{g, h}, E_\bullet, F_\bullet)$. By Theorem \ref{thm:gln_scalable}, $\cK(T)$ is precisely the convex body 
$$\left\{(\textbf{p}, \textbf{q}): \sum \Delta p_i c_i + \Delta q_i d_i \leq \Tr P \textrm{ for all } (\mathbf{c}, \mathbf{d}) \in D'\right\}.$$

\end{proof}


\subsection{Correctness of Algorithm \ref{alg:informal_scaling}}
Here we show that Algorithm \ref{alg:informal_scaling} works in polynomially many steps. This essentially follows from the Schwarz-Zippel lemma combined with the Derksen's degree bound in Theorem \ref{thm:derksen}.
\begin{proposition}\label{prp:toy_works_app}
Suppose $T, P, Q$ have bit-complexity at most $b$ and that $T$ is approximately $\GL_m(\CC)\times \GL_n(\CC)$-scalable to $(I_m \to Q, I_n \to P)$. If $\epsilon < 1$ and
$$TIME \geq \frac{400 b m}{\min\{q_m, p_n\} \epsilon^2}$$ then Algorithm \ref{alg:informal_scaling} outputs ERROR with probability at most 1/3.
\end{proposition}
\begin{proof}
Suppose $T$ is approximately $\GL_m(\CC)\times \GL_n(\CC)$-scalable to $(I_n \to Q, I_m \to P)$. By Theorem \ref{thm:gln_scalable}, $T_{g_0,h_0}$ is $(P,Q)$-rank-nondecreasing for generic $(g_0^\dagger, h_0)\in \GL_m(\CC)\times \GL_n(\CC)$. Since $P,Q$ have bit-complexity at most $b$, there is a number $\gamma \leq 2^{b}$ such that $\gamma P, \gamma Q$ have integral spectra. In particular, by Lemma \ref{lem:exists_iff_generic_stoc}, $$V' = V(T,P,Q) \cup \{B:\det B = 0\} \cup \{C:\det C = 0\}$$
is an affine algebraic variety in $\Mat_{m\times m}\CC \times \Mat_{n\times n} \CC$ generated by polynomials of degree at most $\max\{2\gamma^2, n,m\} \leq 2\cdot 2^{2b}$ that does not contain all of $\GL_m(\CC)\times \GL_n(\CC)$.\\
\indent There must be some polynomial $p:\Mat_{m\times m}\CC \times \Mat_{n\times n} \CC \to \CC$ of at degree at most $2 \cdot 2^{2b}$ that vanishes on $V'$ but not on all of $\Mat_{m\times m}\CC \times \Mat_{n\times n} (\CC)$. By the Schwarz-Zippel lemma, $p$ vanishes on our random choice of $(g_0,h_0) \in \Mat_{m\times m}\CC \times \Mat_{n\times n} \CC$ with entries in $6 \cdot 2^{2b}$ with probability at most 1/3. With probability at least 2/3 we have found $g_0,h_0$ such that $T_{g_0,h_0}$ is $(P,Q)$-rank-nondecreasing.\\
\indent The bit-complexity of the entries of $T_{g_0, h_0}$ is at most $ \log(m) + \log (n) + 3b \leq 4b$.  The rest of the algorithm is the same as Algorithm \ref{alg:alg_g_tri}, except we accept only when $T_{g_j, h_j}$ is an $\epsilon$-$(P \to I_m, Q \to I_n)$-scaling of $T$, rather than the less stringent requirement that $\ds_{P,Q} T_{g_j, h_j} < \epsilon$. However, as we remarked in Definiton \ref{dfn:ds_def}, if $\ds_{P,Q} T_{g_j, h_j} < \epsilon' = \min\{p_n,q_m\}\epsilon^2$ then $T_{g_j, h_j}$ is an $\epsilon$-scaling of $T$. This gives the required upper bound on $TIME$.
\end{proof}

\begin{remark}[Numerical issues]\label{rem:bit_complexity} We should not expect to be able to compute each step of Algorithm \ref{alg:alg_g_tri} exactly, but rather to polynomially many bits of precision. The previous version of this paper had a rather messy analysis of the rounding; \cite{BFGOWW18} contains a much more pleasant analysis, which we now sketch.\\
\indent In each iteration of Algorithm \ref{alg:alg_g_tri}, simply compute $g$ (resp. $h$) to some precision $2^{-t}$ (ensuring that they are upper triangular). We need to check that progress is still made per step and that the capacity is bounded above the entire time. For this we need to emulate the proof of Lemma \ref{lem:cap_incr} with some error; it is enough to show that the rounded $h'$ satisfies $1/\det(Q, h^{-\dagger} h^{-1}) \approx 1/\det(Q, h'^{-\dagger} h'^{-1}).$ Since $h^{-\dagger} h^{-1} = T^*(Q)$, it is enough to show that the least eigenvalues of $T^*(Q)$ (resp $T(P)$) stays bounded singly exponentially away from zero throughout. If this holds, the capacity is also bounded because $T(P) \approx I$ or $T^*(Q) \approx I$. Fortunately, $t$ can be chosen $\poly(TIME, p_n^{-1}, q_m^{-1}, m,n, b)$ such that the least eigenvalues do stay large enough throughout.

\end{remark}

\section{Extension and special cases}\label{sec:applications}
Here we discuss a few of the questions in Section \ref{subsec:special_cases}. First, we find that all of the special cases have a certain structure which resembles a blown-up version of matrix scaling. There's nothing particularly special about this structure, but it allows one to apply Theorem \ref{thm:gln_scalable} much more easily.
\begin{definition}\label{dfn:block_diag}
Let $\mathbf{n} = (n_1, \dots, n_s)$ be a sequence of positive integers summing to $n$, and $\mathbf{m} = (m_1, \dots, m_t)$ a sequence of positive integers summing to $m$. Say $T$ is $(\mathbf{m}, \mathbf{n})$-block-diagonal if every Kraus operator $T$ is of the form 
\begin{align*}
B_{i;j,k} = \left[\begin{array}{ccccc} 0_{m_1, n_1} & \hdots &  0_{m_1, n_k} & \hdots& 0_{m_1, n_{s}} \\
\vdots &\ddots&\vdots &\ddots&\vdots\\
0_{m_j, n_1}  & \hdots&  A_{i;j,k} &\hdots &0_{m_j, n_{s}} \\
 \vdots &\ddots&\vdots &\ddots&\vdots\\
0_{m_{t}, n_1}  & \hdots & 0_{m_{t}, n_k} & \hdots & 0_{m_{t}, n_{p_1}}
\end{array}\right]. 
\end{align*}
where $A_{i; j,k} \in \Mat_{m_j \times n_k}\CC$, $i \in [r]$, $j \in [t]$, $k \in [s]$. Define 
$$\GL_{\mathbf{m}}(\CC) = \bigoplus_{i = 1}^t \GL_{m_i}(\CC)$$
 and $\GL_{\mathbf{n}}(\CC) = \bigoplus_{j = 1}^s \GL_{n_j}(\CC)$. For $i \in [t]$, define
 $E_\bullet(i)$ to be the standard flag on $\CC^{m_i}$, $\mathbf{q}(i)$ to be a nondecreasing sequence of $m_i$ nonnegative numbers, $Q(i) = \diag(\mathbf{q}(i))$, and $Q = \oplus_{i = 1}^t Q(i)$. Define $F_\bullet(i)$, $\mathbf{p}(j)$, $P(j)$, and $P$ analogously.
\end{definition}
We can phrase Question \ref{amcos_four_stoc} in a more convenient way for the purpose of the reduction.
\begin{qn} \label{qn:block_diag}
Which $(\mathbf{m}, \mathbf{n})$-block-diagonal completely positive maps $T$ are approximately $\GL_{\mathbf{m}}(\CC) \times \GL_{\mathbf{n}}(\CC)$ scalable to 
$$\left( I_n \to Q, I_m \to P\right)?$$
\end{qn}
It is convenient to define a more restricted notion of $T$-independence.
\begin{definition}\label{def:block_independent}
Suppose $T$ is $(\mathbf{m}, \mathbf{n})$-block-diagonal. Say the pair of tuples $\mathbf{L} = (L_1, \dots, L_t)$ and $\mathbf{R} = (R_1, \dots, R_s)$ of subspaces $L_i \subset \CC^{m_i}$ and $R_j \subset \CC^{n_j}$ are block-$T$-independent if $L_j \subset (A_{i; j,k} R_k)^\perp$ for all $i \in [r], j \in [t], k \in [s]$. 
\end{definition}
We omit the proof of the following claim, which is tedious but straightforward linear algebra.
\begin{proposition}\label{thm:bl_rnd}
Suppose $T$ is $(\mathbf{m}, \mathbf{n})$-block-diagonal. Then
 $T$ is $(P, Q)$-rank-nondecreasing if and only if $\Tr P = \Tr Q$ and for every block-$T$-independent pair $(\mathbf{L}, \mathbf{R})$, 
\begin{align} \sum_{i = 1}^t \sum_{j = 1}^{m_i} \Delta q_j(i) \dim L_i \cap E_j(i) + \sum_{i = 1}^s \sum_{j = 1}^{n_i} \Delta p_j(i) \dim R_i \cap F_j(i) \nonumber\\
 \leq \Tr P.\label{eq:bl_rnd}
\end{align}
\end{proposition}
\begin{proof} Clearly, \ref{eq:bl_rnd} holds if $T$ is $(P, Q)$-rank-nondecreasing, because if $(\textbf L, \textbf R)$ is block-$T$-independent then $\oplus_i L_i$, $\oplus_i R_j$ are $T$-independent and the value of the left-hand side of \ref{diff_rnd_eq} on the standard flags (with basis vectors added in an order under which $P,Q$ have decreasing diagonal) is the left-hand side of \ref{eq:bl_rnd}.\\ 

It remains to show that if \ref{eq:bl_rnd} holds then $T$ is $(P, Q)$-rank-nondecreasing; it is enough to show show that we only need to check \ref{diff_rnd_eq} on subspaces of the form $R = \oplus R_i$ and $L = \oplus L_i$. Consider $L$ and $R$ that are $T$-independent. We may assume $L$ is maximal for $R$ fixed and $R$ is maximal for $L$ fixed. Note that if $\pi_i: \CC^n \to \CC^n$ is the projection to the $i^{th}$ summand (isomorphic to $\CC^{n_i}$), then $R' = \sum_i  \pi_i R \supset R$ and $(L,R')$ is also $T$-independent. Thus, we may assume $R = \oplus R_i$ and $L = \oplus L_i$.
\end{proof}
The proofs of the next theorem closely mirrors the proof of Theorem \ref{thm:gln_scalable}. The only difference is that the scalings must be taken in $\GL_{\mathbf{m}}(\CC) \times \GL_{\mathbf{n}}(\CC)$ - however, if $T$ is $(\mathbf{m}, \mathbf{n})$-block-diagonal, it's easy to show this is always possible. One must also use that the reduction to the nonsingular case $T \mapsto \underline{T}$ of Lemma \ref{inv_red_lite} preserves being block-diagonal but possibly decreases $m_i$ and $n_j$.
\begin{theorem}\label{thm:block_diag_scalable}
If $T$ is $(\mathbf{m}, \mathbf{n})$-block-diagonal, then $T$ is $\GL_{\mathbf{m}}(\CC) \times \GL_{\mathbf{n}}(\CC)$-scalable to $\left( I_n \to Q, I_m \to P\right)$
if and only if $T_{\mathbf{g}, \mathbf{h}}$ is $(P,Q)$-rank-nondecreasing for generic $(\mathbf{g}^\dagger, \mathbf{h}) \in \GL_{\mathbf{m}}(\CC) \times \GL_{\mathbf{n}}(\CC)$.
\end{theorem}

\subsection{Matrix Scaling}\label{rc_scaling_app}

 Say $XAY$ is an \emph{$\epsilon$-$(\mathbf{r},\mathbf{c})$-scaling} of a nonnegative matrix $A$ if $X$ and $Y$ are diagonal matrices and the and column sum vectors of $XAY$ are at most $\epsilon$ from $\mathbf{r}$ and $\mathbf{c}$, respectively, in (say) Euclidean distance and that $A$ is \emph{approximately $(\mathbf{r},\mathbf{c})$-scalable} if for every $\epsilon>0$ there exists an $\epsilon$-$(\mathbf{r},\mathbf{c})$-scaling of $A$. Given $A, r, c$, the $(\mathbf{r},\mathbf{c})$-scaling problem consists of deciding the existence of and finding $\epsilon$-$(\mathbf{r},\mathbf{c})$-scalings. The $(\mathbf{r},\mathbf{c})$-scaling problem has practical applications such as statistics, numerical analysis, engineering, and image reconstruction, and theoretical uses such as strongly polynomial time algorithms for approximating the permanent \cite{Si64}, \cite{RS89},\cite{LSW98}.\\
\indent There is a simple criterion for approximate $(\mathbf{r},\mathbf{c})$-scalability. 
\begin{theorem}[Rothblum and Schneider \cite{RS89}] \label{rcthm}
A nonnegative matrix $A$ is approximately $(\mathbf{r},\mathbf{c})$ scalable if and only if $\sum_{i} r_i = \sum_{j} c_j$ and for every zero submatrix $L\times R$ of $A$, 
$$\sum_{i\in L} r_i \leq \sum_{j \not\in R} c_j.$$
\end{theorem}
We can reduce this to an instance of Question \ref{amcos_four_stoc} as follows:
\begin{definition}\label{rc_op}
Suppose $A$ is a nonnegative $m \times n$ matrix. For $i \in [m], j \in [n]$, define $e_{ij}$ to be the $m\times n$ matrix with a one in the $ij$ entry and zeros elsewhere. Let $T_A: \Mat_{n\times n}\CC \to \Mat_{m\times m}\CC$ be the completely positive map with Kraus operators $E_{ij} = \sqrt{A_{ij}}e_{ij}$, $i \in [m], j \in [n]$.
\end{definition}
If $\mathbf{m} = (1, \dots, 1)$ and $\mathbf{n} = (1, \dots, 1)$, then $T_A$ is $(\mathbf{m}, \mathbf{n})$-block-diagonal. The next proposition is easy to check.
\begin{proposition}\label{prp:rc_scal}
$A$ is approximately $(\mathbf{r},\mathbf{c})$-scalable if and only if $T_A$ is approximately $\GL_{\mathbf{m}}(\CC) \times \GL_{\mathbf{n}}(\CC)$-scalable to 
$$(I_n \to  \diag(\mathbf{r}), I_m \to\diag(\mathbf{c})).$$
\end{proposition}
Theorem \ref{rcthm} follows easily from Theorem \ref{thm:block_diag_scalable}.
\begin{proof}[Proof of Theorem \ref{rcthm}]
By Proposition \ref{prp:rc_scal}, it is enough to characterize $(\diag(\mathbf{r}), \to\diag(\mathbf{c}))$-rank-nondecreasingness of $T_A$. Since $\CC^{m_i}$ and $\CC^{n_i}$ are copies of $\CC$, $L_i, R_j$ are either $\{0\}$ or $\CC$, $E_\bullet(i) = (\{0\}, \CC)$ and $F_\bullet(j) = (\{0\}, \CC)$. Let $L = \{i: L_i = \CC\} \subset [m]$ and $R = \{j: R_j = \CC\} \subset [n]$. Note that $\mathbf{L}$ and $\mathbf{R}$ are block-$T_A$-independent if and only if $T\times L$ is a zero submatrix of $A$, and the group elements $\mathbf{g}$ and $\mathbf{h}$ make no difference. Equation \ref{eq:bl_rnd} becomes 
$$ \sum_{i \in L} c_i + \sum_{j \in R} r_i \leq \sum_{j \in [n]} r_i,$$
which is equivalent to the condition for approximate $(\mathbf{r},\mathbf{c})$-scalability in \ref{rcthm}.
\end{proof}

\subsection{Eigenvalues of Sums of Hermitian Matrices}\label{hermitians_app}
Here is an old question in linear algebra, apparently originally due to Weyl. It is also sometimes called Horn's Problem.
\begin{qn}[Weyl]\label{gelf}
Let $\alpha, \beta, \gamma$ be nonincreasing sequences of $m$ real numbers. When are $\alpha, \beta, \gamma$ the spectra of some $m\times m$ Hermitian matrices $A,B,C$ satisfying $A + B = C$?
 \end{qn}
This question essentially asks for a complete list of inequalities satisfied by the eigenvalues of sums of Hermitian matrices. Klyachko showed a relationship between the eigenvalues of sums of Hermitian operators and certain constants known as the \emph{Littlewood-Richardson coefficients}. 
\begin{definition}
If $I = \{i_1 < \dots< i_k\}\subset [m],$ let $\rho(I)$ be the partition
$$\rho(I) = (i_k - k, \dots, i_2 - 2, i_1 - 1).$$
\end{definition}
The Littlewood-Richardson coefficient of the partitions $\lambda, \mu$, and $\nu$ is a nonnegative integer denoted $c^\lambda_{\mu,\nu}$. 
\begin{theorem}[Klyachko \cite{Kl98}]\label{kly}
The three nonincreasing sequences $\alpha, \beta, \gamma$ of length $m$ are the spectra of some $m\times m$ Hermitian matrices $A,B,C$ satisfying $A + B = C$ if and only if $\sum_{i = 1}^m \alpha_i + \beta_i -\gamma_i = 0$
and for all $n < m$, 
$$\sum_{i \in I} \alpha_i + \sum_{j \in J} \beta_j \leq \sum_{k \in K} \gamma_k$$ 
for all $|I| = |J| = |K| = n$ such that the Littlewood-Richardson coefficient $c^{\rho(K)}_{\rho(I), \rho(J)}$ is positive.
\end{theorem}
Though computing $c^\nu_{\lambda, \mu}$ is $\#P$-hard, there exists an algorithm to decide if $c^\nu_{\lambda, \mu} > 0$ in strongly polynomial time \cite{MNS12}.\\
\indent Combined with Theorem \ref{kly}, Knutson and Tao's answer to the Saturation conjecture in the positive \cite{KT00} and a different result of Klyachko \cite{Kl98} show that the admissible spectra are described by a recursive system of inequalities originally conjectured by Alfred Horn \cite{H62}. \\
\indent We show that Question \ref{gelf} can be reduced to an instance of Question \ref{amcos_four_stoc}, after which Theorem \ref{kly} will be a corollary of our main theorem. This results in an algorithmic proof of Theorem \ref{kly}. 
We can restate Question \ref{gelf} for more than $3$ matrices. We will instead search for tuples of matrices with given spectra that add to a multiple of the identity; this is equivalent because we may subtract the target matrix from both sides and then add suitable scalar multiples of the identity to left and right hand side. This also shows it is enough to find the matrices when all the spectra are positive. 

\begin{qn}\label{qn:more_mats}
For which tuples $(\mathbf{p}(1) \dots \mathbf{p}(s))$ of length-$m$ weakly decreasing sequences positive numbers do there exist positive-definite $m\times m$ Hermitian matrices $H_i$, $i \in [s]$ such that $\lambda(H_i) = \mathbf{p}(s)$ and 
$$\sum_{i=1}^t H_i =  I_m?$$
\end{qn}
Klyachko answered Question \ref{qn:more_mats} in terms of intersections of Schubert varieties with respect to generic flags, which can in turn be described by the combinatorially defined higher Littlewood-Richardson coefficients $c^{\lambda}_{\mu_1, \dots \mu_r}$. Here is the statement we wish to reprove via operator scaling. We do not show that it is the same as Theorem \ref{kly}, but instead refer the reader to \cite{F00}. Let $\mathbf{n} = (m, \dots, m)$, be a sequence of $s$ many $m's$. Let $\GL_{\mathbf{n}}(\CC)$ and $F_\bullet(i), i \in [s]$ be as in Definition \ref{dfn:block_diag}.

\begin{theorem}[Klyachko, \cite{Kl98}]\label{kly_more}
The answer to Question \ref{qn:more_mats} is positive if and only if $\Tr P = m$ and for a generic tuple $$\mathbf{h} = (h(1), \dots, h(s)) \in \GL_{\mathbf{n}}(\CC),$$
the inequality 
\begin{align}\sum_{i = 1}^s \sum_{j = 1}^n \Delta p_j(i) \dim R \cap h(i) F_j(i) \leq \dim R\label{eq:kly_rnd}
\end{align}
holds for all $R \subset \CC^m$. \end{theorem}
We now reduce Question \ref{qn:more_mats} to Question \ref{qn:block_diag}.
\begin{definition} Let $n = ms$, and again let $\mathbf{n} = (m, \dots, m)$, be a sequence of $s$ many $m's$, and let $\mathbf{m} = (m)$.
Define $T_m^s:\Mat_{n\times n} \CC \to \Mat_{m\times m} \CC$ to be the $(\mathbf{m}, \mathbf{n})$-block-diagonal competely positive map with Kraus operators 
$$  A_i:= \left[\begin{array}{ccc}  0_{m,mi-m} &  I_m  & 0_{m,mi}  \end{array}\right] $$
for $i \in [s]$. 
\end{definition}
\begin{proposition}\label{horn_red}
The answer to Question \ref{qn:more_mats} is positive for the tuple
 $(\mathbf{p}(1) \dots \mathbf{p}(s))$ if and only if $T_m^s$ is approximately $\GL_{\mathbf{m}}(\CC)\times \GL_{\mathbf{n}}(\CC)$-scalable to $( I_{n} \to I_m, I_{m} \to P)$.
\end{proposition}
\begin{proof}Set $T:=T_m^s$.
First we prove the ``only if" statement. Suppose there exist $H_1, \dots, H_r$ with $\lambda(H_i) = \mathbf{p}(i)$ and $\sum_i H_i = I_m$. As $H_i \succeq 0$, we can write $H_i = B_i B_i^\dagger$ where $B_i^\dagger B_i = P(i).$ This is because $B_i B_i^\dagger$ and $B_i^\dagger B_i$ have the same spectrum and  $B_i B_i^\dagger$ is invariant under $B_i \to B_i U_i$ for $U_i$ unitary. Since $B_i$ is invertible for $i \in [r]$, it follows that $\mathbf{h} = \bigoplus_{i \in [r]} B_i \in \GL_{\mathbf{n}}(\CC)$ and 
$$T_{I, \mathbf{h}}(I_{n}) = \sum_{i = 1}^s B_i  B_i^\dagger = I_m \textrm{ and } (T_{I,\mathbf{h}})^*(I_m) = \oplus_{i = 1}^s B_i^\dagger B_i = P.$$
so $T$ is approximately $\GL_{\mathbf{m}}(\CC)\times \GL_{\mathbf{n}}(\CC)$ scalable to $(P,I)$.
The ``if" direction is also easy; suppose $(\mathbf{g}_k, \mathbf{h}_k)$ is a sequence of elements of $G\times H$ such that 
$$T_{\mathbf{g}_k, \mathbf{h}_k}(I_{n}) \to I_m \textrm{ and } (T_{\mathbf{g}_k, \mathbf{h}_k})^*(I_m) \to P$$
as $k\to \infty$
and that $\mathbf{h}_k = \oplus_{i \in [s]} h_k(i).$ Set $B_k(i) = \mathbf{g}_k^\dagger h_k(i)$; thus, we have 
$$ \sum_{i \in [s]} B_k(i) B_k(i)^\dagger \to I_m$$
and, for all $i \in [s]$, 
$$B_k(i)^\dagger B_k(i) \to P(i).$$  
Since the $B_k(i) B_k(i)^\dagger$ are positive definite, eventually for all $i$,\\
$B_k(i) B_k(i)^\dagger$ remains in the compact set $\{X: 0 \leq X \leq 2I_m\}$. Thus, we may pass to a subsequence such that for all $i \in [s]$ we have $B_k(i) B_k(i)^\dagger \to H_i$; by continuity the $H_i$ satisfy 
$$\sum_{i = 1}^s H_i = I_m.$$
and $\lambda(H_i) = \mathbf{p}(i)$ for $i \in [s]$.\end{proof}
This and Proposition \ref{prp:toy_works_app} prove Proposition \ref{prp:inverse_eig}, i.e. that Algorithm \ref{alg:inverse_eig} runs in time $b m^2/\epsilon$.\\
\indent The proof of Theorem \ref{kly_more} is now immediate:
\begin{proof}[Proof of Theorem \ref{kly_more}]
By Theorem \ref{thm:block_diag_scalable} and Proposition \ref{horn_red}, the answer to Question \ref{qn:more_mats} is positive if and only if $(T_m^s)_{\mathbf{g}, \mathbf{h}}$ is $(P,Q)$-rank-nondecreasing for generic $(\mathbf{g}^\dagger, \mathbf{h}) \in \GL_{\mathbf{m}}(\CC) \times \GL_{\mathbf{n}}(\CC)$.\\
By Theorem \ref{thm:bl_rnd}, this is true if and only if for a generic $(\mathbf{g}^\dagger, \mathbf{h}) \in \GL_{\mathbf{m}}(\CC) \times \GL_{\mathbf{n}}(\CC)$, we have 
\begin{align} \dim L + \sum_{i = 1}^s \sum_{j = 1}^{m} \Delta p_j(i) \dim R_i \cap F_j(i) 
 \leq m.\label{eq:weyl_rnd}
\end{align}
for all $\mathbf{R} = (R_1, \dots, R_s)$ and $\mathbf{L} = (L)$ such that $(\mathbf{L},\mathbf{R})$ is block-$(T_m^s)_{\mathbf{g}, \mathbf{h}}$-independent. Equivalently, the tuples given by $\mathbf{L} ' = (g L)$ and $\mathbf{R}'_i = (h(i) R_i: i \in [s])$ are $T_m^s$-independent. Eq. \ref{eq:weyl_rnd} becomes 
\begin{align} \dim L' + \sum_{i = 1}^s \sum_{j = 1}^{m} \Delta p_j(i) \dim R_i' \cap h(i) F_j(i)
 \leq m.\label{eq:weyl_rnd_1}
\end{align}
Finally, if $(\mathbf{L}', \mathbf{R}')$ is block-$T$-independent, we may replace $R'_i$ by $R = \sum_{i \in [s]}{R_i}$ and $L'$ by $R^\perp$ while only increasing Eq. \ref{eq:weyl_rnd_1} so that Eq. \ref{eq:weyl_rnd_1} becomes Eq. \ref{eq:kly_rnd}.
\end{proof}


\subsection{Extensions of Theorems of Barthe and Schur-Horn}\label{forst_app}

Let $\mathbf{U} = (u_1, \dots, u_n)\in (\CC^m)^n$ be an ordered tuple of complex $m$-vectors, and $\mathbf{p} = (p_1, \dots, p_n) \in \RR_{>0}^n$. Say a linear transformation $B:\CC^m \to \CC^m$ puts a collection of vectors $\mathbf{U}$ in \emph{radial isotropic position} with respect to $\mathbf{p}$ if  $$\sum_{i =1}^{n}p_i\frac{ Bu_i ( Bu_i)^\dagger}{\|Bu_i\|^2} = I.$$
\begin{qn}\label{forst_quest}
Given $\mathbf{U}$ and $\mathbf{p}$, when is there a linear transformation $B$ that puts $\mathbf{U}$ in isotropic position with respect to $\mathbf{p}$? 
\end{qn}

Barthe showed Question \ref{forst_quest} has a positive answer if $p$ lies in a certain polytope, which we now describe.


\begin{definition}\label{forster_poly}Let $\mathbf{U} = (u_1, \dots, u_n)\in (\CC^m)^n$ be an ordered tuple of complex $m$-vectors. Let $\mathbf{B}\subset \binom{[n]}{k}$ be the collection of $m$-subsets $S$ of $[n]$ such that $\{u_i:i \in S\}$ forms a basis of $\CC^m$. If $\mathbf{q} = (q_1, \dots, q_m)$ is a sequence of nonnegative numbers, define
$$\cK_\mathbf{q}(\mathbf{U}) = \operatorname{conv}\{(1_S(i) q_\sigma(i):i \in [n] ): S \in \mathbf{B}, \sigma: S \leftrightarrow [m]\}$$
Informally, each vertex of the polytope is the indicator vector for each basis in $\mathbf{U}$ with the nonzero entries replaced by $q_1, \dots q_m$ in some order. If $\mathbf{q}$ is the all-ones vector, $B(\mathbf{U}):=\cK_{\mathbf{q}}(\mathbf{U})$ is known as the \emph{basis polytope}. 
\end{definition}

\begin{theorem}[Barthe \cite{B98}]\label{b98} 
$\mathbf{p} \in B(\mathbf{U})$ if and only if there are linear transformations $B$ that put $\mathbf{U}$ arbitrarily close to radial isotropic position with respect to $\mathbf{p}$.\\

Further, if $p$ is in the relative interior of $B(\mathbf{U})$, then there are linear transformations $B$ that put $\mathbf{p}$ in radial isotropic position with respect to $\mathbf{p}$.
\end{theorem}

As a partial answer to Question \ref{forst_quest_stoc}, we prove a generalization of Barthe's Theorem.
\begin{definition}
Say \emph{$\mathbf{U}$ can be approximately put in $Q$-isotropic position with respect to $\mathbf{p}$} if for every $\epsilon > 0$ there exists an invertible linear transformation $B$ 
such that 
\begin{equation}\left\|\sum_{i =1}^{n}p_i\frac{  B u_i  (B u_i)^\dagger }{\| Bu_i\|}  - Q\right\| \leq \epsilon.\label{GenGeometric}\end{equation}
\end{definition}

\begin{theorem}\label{gen_forst_thm} 
Suppose $Q$ is a positive-definite matrix with spectrum $\mathbf{q} = (q_1, \dots, q_m)$. $\mathbf{U}$ can be approximately put in $Q$-isotropic position with respect to $\mathbf{p}$ if and only if $\mathbf{p} \in \cK_\mathbf{q}(\mathbf{U})$.
\end{theorem}

Theorem \ref{gen_forst_thm} follows immediately from the next two propositions. The first follows from Theorem \ref{thm:gln_scalable}, and the second from Edmonds' work on polymatroids.
\begin{proposition}\label{alt_char_forst}
$\vec U$ can be approximately put in $Q$-isotropic position with respect to $\vec p$
if and only if 
$$\sum_{j = 1}^n p_j = \sum_{i = 1}^m q_i $$
and
$$\sum_{j \in J} p_j \leq \sum_{i = 1}^{\dim \langle u_j: j \in J\rangle} q_i$$
for all $J \subset [n]$.
\end{proposition}

\begin{proposition}\label{alt_char_poly} $\vec p$ is in $\cK_{\vec q}(\vec U)$ if and only if
\begin{equation}\sum_{j = 1}^n p_j = \sum_{i = 1}^m q_i \label{tr_forst}\end{equation} and 
\begin{equation}\sum_{j \in J} p_j \leq \sum_{i = 1}^{\dim \langle u_j: j \in J\rangle} q_i \label{submod_forst}\end{equation}
for all $J \subset [n]$. 

\end{proposition}

\begin{proof}[Proof of Proposition \ref{alt_char_poly}]
Polytopes of the form $\sum_{j \in J} p_j \leq f(J)$ for submodular set functions $f$ are well-understood, so we first check that our constraints take this form.
\begin{lemma}\label{submodlem}
The function  \begin{equation}f_{\vec q}(J) =\left\{ \begin{array}{cc} \label{submod}\sum_{i =1}^{\dim\langle u_j: j \in J \rangle} q_j & J \neq \emptyset\\
f_{\vec q}(J)= 0 & J = \emptyset\end{array} \right.
\end{equation}
is a nonnegative, monotone, submodular function on the lattice of subsets of $[n]$. 
\end{lemma}
\begin{proof}[Proof of Lemma \ref{submodlem}] Nonnegativity and monotonicity are clear. To show that $f_{\vec q}$ is submodular, it is enough to show $f_{\vec q}$ gives decreasing marginal returns, that is, for $X \subset Y$ and $x \in [n] \setminus Y$, 
$$ f_{\vec q}(X \cup \{j\}) - f_{\vec q}(X) \geq f_{\vec q}(Y \cup \{j\}) - f_{\vec q}(Y).$$
Indeed, 
\begin{eqnarray*} f_{\vec q}(X \cup \{j\}) - f_{\vec q}(X) = \sum_{i =1}^{\dim(\langle u_i: i \in X \rangle + u_j ) } q_i - \sum_{i =1}^{\dim\langle u_i: i \in X \rangle } q_i\\
= 1_{u_j \not\in \langle u_i: i \in X \rangle} q_{\dim\langle u_i: i \in X \rangle + 1} \geq 1_{u_j \not\in \langle u_i: i \in Y \rangle} q_{\dim\langle u_i: i \in Y \rangle + 1}\\
=f_{\vec q}(Y \cup \{j\}) - f_{\vec q}(Y).
\end{eqnarray*}
\end{proof}
Next we use a theorem of Edmonds. 
\begin{theorem}[\cite{Ed70}] If $E$ is a finite set and $L$ is an intersection-closed family of subsets of $E$, let 
$$P(E,f) = \{x \in \RR^E_+: \forall S \in L - \emptyset,\; \sum_{i \in S} x_i \leq f(S) \}.$$
If $f$ is a nonnegative, monotone function on $2^{E}$ with $f(\emptyset) = 0$, then each vertex $x$ of $P(E,f)$ is given by
$$ x_\sigma(i) = f(\{\sigma(j): j \leq i \}) - f(\{\sigma(j):j < i\}) $$
for some ordering $\sigma: [|E|] \leftrightarrow E,$ and every ordering corresponds to such a vertex.
\end{theorem}

Next note that $p$ satisfies the equality \ref{tr_forst} and the inequality \ref{submod_forst} if and only if $$p \in \cK'_{\vec q}(\vec U) :=P([n], f_{\vec q}) \cap \left\{p: \sum_{i = 1}^n p_i = \sum_{i = 1}^m q_i \right\}.$$ 
For $x \in P([n],f_{\vec q})$, $\sum_{i = 1}^n x_i \leq \sum_{i = 1}^m q_i$ by the constraint when $J = [n]$. Thus, $\cK'_{\vec q}(\vec U)$ is the convex hull of the vertices of $P([n],f_\lambda)$ that are contained in the hyperplane $\left\{p: \sum_{i = 1}^n p_i = \sum_{i = 1}^m q_i \right\}$. Recall from Definition \ref{forster_poly} that these are exactly the vertices of $\cK_{\vec q}(\vec U)$; hence $\cK_{\vec q}(\vec U) = \cK'_{\vec q}(\vec U)$. The proposition is proved.
\end{proof}

\begin{proof}[Proof of Lemma \ref{alt_char_forst}] Consider the completely positive map $T_{\vec U}:\Mat_{n\times n}(\CC) \to \Mat_{m\times m}(\CC)$ with Kraus operators 
\begin{align*} A_i:= \left[\begin{array}{ccc}  0_{m,i-1} &  u_i  & 0_{m,n-i}  \end{array}\right] \end{align*}
for $i \in [n]$. Here $0_{m, k}$ denotes an $m \times k$ zero submatrix. Then it is not hard to see that $\vec U$ can be approximately put in $Q$-isotropic position with respect to $p$ if and only if $T_{\vec U}$ is approximately $(G = \GL_m(\CC), H = \GL_{\vec n}(\CC))$-scalable to $(I_n \to Q, I_m \to P)$, where $\vec n = (1, \dots, 1)$. That is, $H$ is the group of diagonal invertible complex matrices.\\

By Theorem \ref{thm:gln_scalable}, $T_{\vec U}$ is approximately $(G,H)$-scalable to $(P,Q)$ if and only if $(T_{\vec U})_{g,h}$ is $(P,Q)$-rank-nondecreasing for a generic $(g^\dagger, h) \in G\times H$. However, multiplication by $h$ doesn't affect $(P,Q)$-rank-nondecreasingness because $h$ is diagonal, so $T_{\vec U}$ is approximately $(G, H)$-scalable to $(P,Q)$ if and only if $(T_{\vec U})_{g,I_n} = T_{g\vec U}$ is $(P,Q)$-rank-nondecreasing for a generic $g \in G$.\\

By Theorem \ref{thm:bl_rnd}, we need only check that $(T_{g\vec U})$ is block-$(P,Q)$-rank-nondecreasing. Since the $R_i$ are one dimensional subspaces, only $S = \{i: L_i = \{0\}\}$ matters. Further, by maximality, we may assume $L = \langle u_i : i \in S \rangle^\perp$. Thus, $T_{g\vec U}$ is $(P,Q)$-rank-nondecreasing if and only if 
$$ \sum_{j \in S} p_j + \sum_{i = 1}^m \Delta q_i \langle gu_i : i \in S \rangle^\perp \cap F_i \leq \Tr Q$$
for all $S \subset [n]$. For generic $g$, the left-hand side of the above inequality is equal to 
$$ \sum_{j \in S} p_j + \sum_{i = \dim \langle u_i : i \in S \rangle^\perp + 1}^m q_i$$
for all $S \subset[ n]$; thus, $T_{g\vec U}$ is $(P,Q)$-rank-nondecreasing for a generic $g \in G$ if and only if 
$$ \sum_{j \in S} p_j \leq \sum_{i = 1}^{\dim \langle u_i: i \in S \rangle} q_i$$
for all $S \subset [n]$.\end{proof}

\begin{remark}[Algorithms] Because $\cK_{\vec q}(\vec U)$ is a polymatroid and we an compute $\sum_{i = 1}^{\dim \langle u_i: i \in S \rangle} q_i$ from $S$ easily, we can easily test if $\vec p \in \cK_{\vec q}(\vec U)$ \cite{Ed70}. However, it is not clear if we can put $\vec U$ in $Q$-isotropic position with respect to $\vec p$ in time $\poly(m, n, b, \log (1/\epsilon))$. The $\poly(m, n, b, \log (1/\epsilon))$ algorithm for the $Q = I$ case in \cite{SV17} suggests this may be possible; one just needs to minimize the convex program 
$$ \inf_{\vec t \in \RR^n}\log \det\left(Q,  \sum_{i = 1}^n e^{t_i} u_i u_i^\dagger \right) - \sum_{i = 1}^n p_i t_i.$$
We wonder if this can also be formulated in terms of relative entropy.
\end{remark}

\indent It's not too hard to see that Theorem \ref{gen_forst_thm} implies the more difficult ``if'' direction of the classic Schur-Horn theorem relating the diagonal and spectra of a Hermitian matrix. 
\begin{theorem}[Schur-Horn \cite{H54}]\label{schur_horn}
There is a Hermitian $n\times n$ matrix with diagonal $p_1 \geq \dots \geq p_n$ and spectrum $q_1 \geq  \dots \geq q_n$ if and only if $q_1, \dots, q_n$ majorizes $p_1, \dots, p_n$. That is, for all $i \leq n$, 
$$\sum_{j = 1}^i p_i \leq \sum_{j = 1}^i q_i.$$
\end{theorem}
To prove the Schur-Horn theorem, simply pick $\mathbf{U}$ to be in general position - note that $\cK_{\mathbf{q}}(\mathbf{U})$ is then the \emph{permutohedron} of the vector $(q_1, \dots, q_m, 0, \dots, 0) \in \RR^n$, which is precisely the set of $\mathbf{p}$ that is majorized by $(q_1, \dots, q_m, 0, \dots, 0)$!



\section{Future Work}
We wonder if in this setting there is an algorithm to find approximate scalings in time polynomial in $-\log(\epsilon)$ rather than $\epsilon^{-1}$. As it is, our algorithm resembles alternating minimization; perhaps other optimization techniques could result in faster algorithms. The recent fast $(\mathbf{r},\mathbf{c})$-scaling algorithms \cite{Li17}, \cite{M17} give hope that this is possible.\\
\indent  Finding a $\poly \log (1/\epsilon)$ algorithm has another benefit: the algorithms herein are not capable of deciding $(P,Q)$-rank-nondecreasingness in strongly polynomial time. By Corollary \ref{cor:sure_rnd}, in order to certify $(P,Q)$-rank-nondecreasingness one requires $\epsilon$-$(I_m \to Q, I_n \to P)$-scalings for $\epsilon$ as small as a common denominator of all the entries of $\mathbf{p}$ and $\mathbf{q}$. However, our algorithm depends polynomially on $\epsilon^{-1}$. In fact, this decision problem was shown to be in $\mathbf{NP} \cap \mathbf{coNP}$ and is conjectured to be in $\mathbf{P}$; at least for \ref{hermitians} there is a strongly polynomial time algorithm to decide if the reduction is $(P,Q)$-rank-nondecreasing that has nothing to do with operator scaling \cite{MNS12}.


\section*{Acknowledgements}
The author would like to thank Michael Saks for many insightful discussions, and Rafael Oliveira for interesting observations and pointers to relevant literature. The author would like to further thank Ankit Garg, Rafael Oliveira, Avi Widgerson and Michael Walter for discussions concerning a forthcoming joint work that greatly simplified the presentation of the reduction to the doubly stochastic case.

\bibliographystyle{plain}
\bibliography{/Users/Cole/Latex/operatorscaling}

\begin{thebibliography}{10}

\bibitem{Li17}
Zeyuan Allen-Zhu, Yuanzhi Li, Rafael Oliveira, and Avi Wigderson.
\newblock Much faster algorithms for matrix scaling.
\newblock {\em arXiv preprint arXiv:1704.02315}, 2017.

\bibitem{B98}
Franck Barthe.
\newblock On a reverse form of the brascamp-lieb inequality.
\newblock {\em Inventiones mathematicae}, 134(2):335--361, 1998.

\bibitem{Bri87}
Michel Brion.
\newblock Sur l'image de l'application moment.
\newblock In {\em S{\'e}minaire d'Alg{\`e}bre Paul Dubreil et Marie-Paule
  Malliavin}, pages 177--192. Springer, 1987.

\bibitem{BCMW17}
Peter B\"urgisser, Matthias Christandl, Ketan~D Mulmuley, and Michael Walter.
\newblock Membership in moment polytopes is in {NP} and {coNP}.
\newblock {\em SIAM Journal on Computing}, 46(3):972--991, 2017.

\bibitem{BFGOWW18}
Peter B{\"u}rgisser, Cole Franks, Ankit Garg, Rafael Oliveira, Michael Walter,
  and Avi Wigderson.
\newblock Efficient algorithms for tensor scaling, quantum marginals and moment
  polytopes.
\newblock {\em arXiv preprint arXiv:1804.04739}, 2018.

\bibitem{M17}
Michael~B Cohen, Aleksander Madry, Dimitris Tsipras, and Adrian Vladu.
\newblock Matrix scaling and balancing via box constrained newton's method and
  interior point methods.
\newblock {\em arXiv preprint arXiv:1704.02310}, 2017.

\bibitem{DM17}
Harm Derksen and Visu Makam.
\newblock Polynomial degree bounds for matrix semi-invariants.
\newblock {\em Advances in Mathematics}, 310:44--63, 2017.

\bibitem{Ed70}
Jack Edmonds.
\newblock Submodular functions, matroids, and certain polyhedra.
\newblock {\em Edited by G. Goos, J. Hartmanis, and J. van Leeuwen}, 11, 1970.

\bibitem{Fo02}
J{\"u}rgen Forster.
\newblock A linear lower bound on the unbounded error probabilistic
  communication complexity.
\newblock {\em Journal of Computer and System Sciences}, 65(4):612--625, 2002.

\bibitem{Fz02}
Matthias Franz.
\newblock Moment polytopes of projective g-varieties and tensor products of
  symmetric group representations.
\newblock {\em J. Lie Theory}, 12(2):539--549, 2002.

\bibitem{Fr16}
Shmuel Friedland.
\newblock On schrodinger's bridge problem.
\newblock {\em arXiv preprint arXiv:1608.05862}, 2016.

\bibitem{F00}
William Fulton.
\newblock Eigenvalues, invariant factors, highest weights, and schubert
  calculus.
\newblock {\em Bulletin of the American Mathematical Society}, 37(3):209--249,
  2000.

\bibitem{GGOWbl16}
Ankit Garg, Leonid Gurvits, Rafael Oliveira, and Avi Wigderson.
\newblock Algorithmic aspects of brascamp-lieb inequalities.
\newblock {\em arXiv preprint arXiv:1607.06711}, 2016.

\bibitem{GGOW16}
Ankit Garg, Leonid Gurvits, Rafael Oliveira, and Avi Wigderson.
\newblock A deterministic polynomial time algorithm for non-commutative
  rational identity testing.
\newblock In {\em Foundations of Computer Science (FOCS), 2016 IEEE 57th Annual
  Symposium on}, pages 109--117. IEEE, 2016.

\bibitem{GP15}
Tryphon~T Georgiou and Michele Pavon.
\newblock Positive contraction mappings for classical and quantum
  schr{\"o}dinger systems.
\newblock {\em Journal of Mathematical Physics}, 56(3):033301, 2015.

\bibitem{Gu04}
Leonid Gurvits.
\newblock Classical complexity and quantum entanglement.
\newblock {\em Journal of Computer and System Sciences}, 69(3):448--484, 2004.

\bibitem{AM13}
Moritz Hardt and Ankur Moitra.
\newblock Algorithms and hardness for robust subspace recovery.
\newblock In {\em Conference on Learning Theory}, pages 354--375, 2013.

\bibitem{H54}
Alfred Horn.
\newblock Doubly stochastic matrices and the diagonal of a rotation matrix.
\newblock {\em American Journal of Mathematics}, 76(3):620--630, 1954.

\bibitem{H62}
Alfred Horn.
\newblock Eigenvalues of sums of hermitian matrices.
\newblock {\em Pacific Journal of Mathematics}, 12(1):225--241, 1962.

\bibitem{Ja74}
A~Jamio{\l}kowski.
\newblock An effective method of investigation of positive maps on the set of
  positive definite operators.
\newblock {\em Reports on Mathematical Physics}, 5(3):415--424, 1974.

\bibitem{Kl98}
Alexander~A Klyachko.
\newblock Stable bundles, representation theory and hermitian operators.
\newblock {\em Selecta Mathematica, New Series}, 4(3):419--445, 1998.

\bibitem{KT00}
Allen Knutson and Terence Tao.
\newblock Honeycombs and sums of hermitian matrices.
\newblock 2000.

\bibitem{Le10}
Veerle Ledoux and Simon~JA Malham.
\newblock Introductory schubert calculus.
\newblock {\em Review Notes, Sept}, 2010.

\bibitem{LSW98}
Nathan Linial, Alex Samorodnitsky, and Avi Wigderson.
\newblock A deterministic strongly polynomial algorithm for matrix scaling and
  approximate permanents.
\newblock In {\em Proceedings of the thirtieth annual ACM symposium on Theory
  of computing}, pages 644--652. ACM, 1998.

\bibitem{MilneAG}
James~S. Milne.
\newblock Algebraic geometry (v6.02), 2017.
\newblock Available at www.jmilne.org/math/.

\bibitem{MNS12}
Ketan~D Mulmuley, Hariharan Narayanan, and Milind Sohoni.
\newblock Geometric complexity theory iii: on deciding nonvanishing of a
  littlewood--richardson coefficient.
\newblock {\em Journal of Algebraic Combinatorics}, 36(1):103--110, 2012.

\bibitem{M99}
D~Mumford.
\newblock The red book of varieties and schemes, second, expanded edition,
  volume 1358 of lnm, 1999.

\bibitem{RS89}
Uriel~G Rothblum and Hans Schneider.
\newblock Scalings of matrices which have prespecified row sums and column sums
  via optimization.
\newblock {\em Linear Algebra and its Applications}, 114:737--764, 1989.

\bibitem{Si64}
Richard Sinkhorn.
\newblock A relationship between arbitrary positive matrices and doubly
  stochastic matrices.
\newblock {\em The annals of mathematical statistics}, 35(2):876--879, 1964.

\bibitem{SV17}
Damian Straszak and Nisheeth~K Vishnoi.
\newblock Computing maximum entropy distributions everywhere.
\newblock {\em arXiv preprint arXiv:1711.02036}, 2017.

\bibitem{CDW13}
Michael Walter, Brent Doran, David Gross, and Matthias Christandl.
\newblock Entanglement polytopes: multiparticle entanglement from
  single-particle information.
\newblock {\em Science}, 340(6137):1205--1208, 2013.

\end{thebibliography}

\appendix

\section{Appendix}\label{app:missing_proofs}
\subsection{Missing proofs for Section \ref{results_sec}}\label{app:results_sec}

\begin{proof}[Proof of Eq. \ref{eq:character}] Let $\textbf a = (a_1 \geq \dots \geq  a_m) \in \RR^m$, $A = \diag (\textbf a)$, and $g_1, g_2 $ upper triangular. Recall that $\det(A, g_1^\dagger g_2) = \prod_{i = m} \det(\eta_i  g_1^\dagger g_2 \eta_i^\dagger)^{\Delta a_i}$. By Proposition \ref{prp:upper_triangular},
\begin{align*}\prod_{i = 1}^m \det(\eta_i  g_1^\dagger g_2 \eta_i^\dagger)^{\Delta a_i} &= \prod_{i = 1}^m \det(\eta_i  g_1^\dagger g_2 \eta_i^\dagger)^{\Delta a_i} \\
&= \prod_{i = 1}^m \det(\eta_i  g_1^\dagger \eta_i^\dagger)^{\Delta a_i} \prod_{i = 1}^m \det (\eta_i g_2 \eta_i^\dagger)^{\Delta a_i}
\end{align*}
Since $ \eta_i  B \eta_i^\dagger$ is just the upper-left $i\times i$ principal minor of a matrix $B$ and $g_2$ is upper triangular, we have 
\begin{align*}
 \prod_{i = 1}^m \det (\eta_i g_2 \eta_i^\dagger)^{\Delta a_i} =  \prod_{i = 1}^m \prod_{j = 1}^i \det (g_2)_{jj}^{\Delta a_i} = \prod_{i = 1}^m (g_2)_{jj}^{a_j} = \chi_{\textbf a} (g_2).
\end{align*}
Deduce by symmetry that also $\prod_{i = 1}^m \det(\eta_i  g_1^\dagger \eta_i^\dagger)^{\Delta a_i} = \overline{\chi_{\textbf a} (g_1)}$, so Eq. \ref{eq:character} holds.
\end{proof}

\subsection{Missing proofs for Section \ref{sec:reduction}}\label{app:reduction}

\begin{proof}[Proof of Item 4 of Theorem \ref{thm:stoc_reds}]
We wish to prove $\trun_{P,Q} T $ is rank-nondecreasing if and only if $T$ is $(P,Q)$-rank-nondecreasing.

Let $N = \Tr P$. First, recall that $\trun_{P,Q} T = G_{\vec q} \circ T \circ G_{\vec p}^*.$ Hence, it is enough to show that $T \circ G_{\vec p}^*$ is $(I_N, S)$-rank-nondecreasing if and only if $T$ is $(P,S)$-rank-nondecreasing for any $S$, because then $T$ is $(P,Q)$-rank-nondecreasing if and only if $T \circ G_{\vec p}^*$ is $(I_N, Q)$-rank-nondecreasing if and only if $G_{\vec p} \circ T^*$ is $(Q, I_N)$-rank-nondecreasing if and only if $G_{\vec p} \circ T^* \circ G_{\vec q}^*$ is $(I_N, I_N)$-rank-nondecreasing.\\

Let $\lambda$ be the conjugate partition to $\vec p$. Recall that the Kraus operators of $T\circ G_{\vec p}^*$ are, for $i \in [r]$ and $j \in [p_1]$, 
$$A_i \eta_{\lambda_j}^\dagger \pi_j : \bigoplus_{k = 1}^{p_1} \CC^{\lambda_k} \to \CC^m$$
 where $\pi_j$ projects to the $j^{th}$ summand. 
 
If $T\circ G_{\vec p}^*$ is $(I_N, S)$ rank-nondecreasing, then $T$ is $(P,S)$-rank-nondecreasing) because for any $T$-independent pair $(L,R)$, \ref{diff_rnd_eq} must hold for $T \circ G_{\vec p}^*, I_N, S$ on the $T\circ G_{\vec p}^*$-independent pair $(L, \bigoplus_{j = 1}^{p_1} R \cap \CC^{\lambda_j})$, which is equivalent to \ref{diff_rnd_eq} holding on $(L, R)$ for $T, P, S$. \\

The more difficult direction is to show that we only need to check \ref{diff_rnd_eq} for $T \circ G_{\vec p}^*, I_N, S$ on the subspaces of the form $(L, \bigoplus_{j = 1}^{p_1} R\cap \CC^{\lambda_j})$. Firstly, using the fact that $T \circ G_{\vec p}^*$ is block-diagonal and applying Proposition \ref{thm:bl_rnd}, we see that it is enough to check $\ref{diff_rnd_eq}$ on $(L, \bigoplus_{j = 1}^{p_1} R_j),$ where $R_j \subset \CC^{\lambda_j}$. Note that $T\circ G_{\vec p}^*$-independence of this pair is equivalent to $(L, R_j)$, viewed as a subspace of $\CC^n$, being $T$-independent for all $j$. Since $T$-independence is closed under addition of when one member of the pair is fixed, and closed downward, we may assume $R_j = R_1 \cap \CC^{\lambda_j}$. \end{proof}

\subsection{Missing proofs for Section \ref{sec:triangular_scalings}}\label{app:triangular_scalings}

\begin{proof}[Proof of Lemma \ref{inv_red_lite}] We first show \ref{cap_lower}. Recall \begin{align*}
\capacity(T, P, Q) = \inf_{h \in \GL(F_\bullet)} \frac{\det(Q, T(hPh^\dagger))}{\det(P, h^\dagger h)}.
\end{align*}
First observe that if $A, B$ are Hermitian, $\det(P, A) = \det(\underline{P}, \eta A \eta^\dagger)$ and similarly $\det(Q, B) = \det(\underline{Q}, \nu B \nu^\dagger)$. Next, by Proposition \ref{prp:upper_triangular}, $T(hPh^\dagger) = T(\eta ^\dagger \underline{hPh}^\dagger \eta)$. Together these observation imply 
$$\capacity(T, P, Q) = \inf_{h \in \GL(F_\bullet)} \frac{\det(\underline{Q}, \underline{T}( \underline{h} \underline{P} \underline{h}^\dagger )  )}{\det(\underline{P},\underline{h}^\dagger \underline{h})};
$$
to complete the proof, observe that $\{\underline{h}: h \in \GL(F_\bullet)\}$ is simply $\underline{H}$.

Let us now prove 2. Suppose $\underline{T}_{\underline{g}, \underline{h}}$ is an $\epsilon$-$(I_{\rank P} \to \underline{Q}, I_{\rank Q} \to \underline{P})$. We use a limiting argument: let $h$ approach $\eta^\dagger \underline{h} \eta$ and $g$ approach $\nu^\dagger \underline{g} \nu $.
Then $g^\dagger T (h h^\dagger) g$ approaches 
\begin{align*}\nu^\dagger \underline{g} \underline{T}( \underline{h} \eta \eta^\dagger \underline{h}^\dagger )   \underline{g}  \nu^\dagger &= \nu^\dagger \underline{g} \underline{T}( \underline{h}  \underline{h}^\dagger )   \underline{g}  \nu^\dagger\\
&= \nu^\dagger (\underline {Q} + X) \nu\\
&= Q  + \nu^\dagger X \nu
 \end{align*}
 Where $X$, and hence $\nu^\dagger X \nu$, has trace-norm at most $\epsilon$. By symmetry, $h^\dagger T^* (g g^\dagger) h$ also approaches a positive-semidefinite matrix at most $\epsilon$ from $P$; this completes the proof. 
\end{proof}

\begin{proof}[Proof of Lemma \ref{lem:dual_inv}]
First we prove the claim for $T$. We can rewrite
\begin{align}
\capacity(T, P, Q) &= \inf_{h \in \GL(F_\bullet)}\frac{ \det(Q, T(h P h^\dagger))}{\det(P, h^\dagger h)}\nonumber\\
&=\inf_{\tilde{h} \in \GL(F_\bullet)}\frac{ \det(Q, T(\tilde{h} \tilde{h}^\dagger))}{\det(P, P^{-1/2}\tilde{h}^\dagger \tilde{h} P^{-1/2})}\nonumber\\
&= \det(P, P) \inf_{h \in \GL(F_\bullet)}\frac{ \det(Q, T(hh^\dagger))}{\det(P, h^\dagger h)}.\label{alt_form_cap}
\end{align}
Since $\det(P,P) > 0$, we have $\capacity(T, P, Q) > 0$ if and only if $\inf_{h \in \GL(F_\bullet)}\frac{ \det(Q, T(hh^\dagger))}{\det(P, h^\dagger h)} > 0$. \\

For any $X \in \cS_{++}(V)$, we can write $X = hh^\dagger$ by the existence of Cholesky decompositions. Since $\capacity(T, P, Q) \neq 0$, $\det(Q,T(hh^\dagger)) >0$. This implies $hh^\dagger$ must be nonsingular because $q_m > 0$. \\

We now prove the claim for $T^*$. Suppose $T^*(Y)$ is singular for $Y \succ 0$. Since $T^*(Y) = \sum_{i = 1}^r A_i^\dagger Y A_i$, $$\ker T^*(Y) \subset \bigcap_i \ker A_i:=R.$$
Notice that $(W, R)$ is a $T$-independent pair. Let $d = \dim R > 0$. For $c > 1$, let $h_c h_c^\dagger = c\pi_R +  \pi_{R^\perp}$. Because $R \subset \ker A_i$, we have $A_i h_ch_c^\dagger A_i^\dagger = A_i \pi_{R^\perp} A_i^\dagger$ for all $i$, or $T(h_c h_c^\dagger) = T(\pi_{R^\perp})$. Then $\det(Q, T(h_c h_c^\dagger)) = \det(Q, T(\pi_{R^\perp}))$.\\
\indent  On the other hand, $h_c^\dagger h_c$ has the same spectrum as $h_c h_c^\dagger$, so it has all eigenvalues at least 1 and an eigenspace of eigenvalue $c$ of dimension at least $d$. Since $P$ is invertible, $p_n > 0$ and 
$$ \det(P, h_c^\dagger h_c) > c^{d p_n}.$$
Plugging $h_c$ into \ref{alt_form_cap} and letting $c \to \infty$ shows $\capacity(T, P, Q) = 0$, a contradiction. \end{proof}

\begin{proof}[Proof of Lemma \ref{lem:cap_evol}]
\begin{align*}
\capacity(T_{g,h}, P,Q) &= \inf_{x \in  \GL(F_\bullet)}\frac{ \det(Q, T_{g,h}(x P x^\dagger))}{\det(P, x^\dagger x)}\\
&=\inf_{x \in  \GL(F_\bullet)}\frac{ \det(Q, g^\dagger T(h x P x^\dagger h^\dagger ))g}{\det(P, x^\dagger x)}\\
&= \inf_{y \in  \GL(F_\bullet)} \frac{ \det(Q, g^\dagger T(y P y^\dagger )g)}{\det(P, y^\dagger h^{-\dagger} h^{-1} y)}\\
&=\inf_{y \in  \GL(F_\bullet)} \frac{ \det(Q, g^\dagger g) \det(Q, T(y P  y^\dagger ))}{\det(P, y^\dagger y)\det(P, h^{-\dagger} h^{-1})}\\
&=\det(Q, g^\dagger g) \det(P, h^\dagger h)  \inf_{y \in  \GL(F_\bullet)} \frac{  \det(Q, T(y P y^\dagger ))}{\det(P, y^\dagger y)}.
\end{align*}
The second two inequalities follow from \ref{twoside_mult} of Lemma \ref{detpx_facts}, and the last from \ref{twoside_inv} of Lemma \ref{detpx_facts}.
\end{proof}

\begin{proof}[Proof of Lemma \ref{lem:cap_incr}]
By \ref{twoside_mult} of Lemma \ref{detpx_facts}, if $h^\dagger T^*(Q) h = I,$ then 
$$\det(P, h^\dagger h) = \frac{1}{\det(P, T^*(Q))}.$$
Thus, it is enough to show 

$$\log \det(P, T^*(Q)) = \log \left(\prod_{i = 1}^n (\det {\eta_i} T^*(Q) {\eta_i}^\dagger)^{\Delta p_i} \right) \leq -   .3 \min\{\epsilon,  p_n\}.$$
For $i \in [n]$ and $j \in [i]$, let $\lambda_{ij}$ be the $j^{th}$ eigenvalue of ${\eta_i} T^*(Q) {\eta_i}^\dagger$. 
Then $$\log \left(\prod_{i = 1}^n (\det {\eta_i} T^*(Q) {\eta_i}^\dagger)^{\Delta p_i} \right) = \sum_{i = 1}^n \Delta p_j \sum_{j = 1}^i \log \lambda_{ij}.$$
Since $\sum_{i = 1}^n i \Delta p_i = \Tr P = 1$, we may define a discrete random variable $X$ by assigning probability $a_i$ to $\lambda_{ij}$. Then 
$$\E[X] = \sum_{i = 1}^n\Delta p_i \Tr {\eta_i} T^*(Q) {\eta_i}^\dagger =  \sum_{i = 1}^n \Tr \Delta p_i {\eta_i}^\dagger{\eta_i} T^*(Q)  = \Tr PT^*(Q) = \Tr T(P) Q = \Tr Q = 1,$$ and it is enough to show $\E[\log X] \leq - \min\{.3 \epsilon, .3 p_n\}$. By definition, 
$$\ds_{P,Q}{T} = \sum_{i = 1}^n \Delta p_i \Tr ({\eta_i} T^*(Q) {\eta_i}^\dagger - I_{i})^2 = \sum_{i = 1}^n \Delta p_i\sum_{j = 1}^i (\lambda_{ij} - 1)^2 = \V[X] \geq \epsilon.$$ 
If a concave function has high variance, then the expectation of the function should be strictly less than the function of the expectation. However, $X$ may have outliers which limits our ability to use this to our advantage. We split into the case where all $\lambda_{ij} \leq 2$ and the case where there is some $\lambda_{ij} > 2$.\\
\indent Define $\epsilon_1 = \V[X| X \leq 2]\Pr[X \leq 2]$, and $\epsilon_2 = \V[X| X > 2]\Pr[X > 2]$ so that $\V[X] = \epsilon_1 + \epsilon_2 \geq \epsilon$. 
If $0 < x\leq 2$, then $\log x \leq (x-1) -.3(x-1)^2$, and by convexity for $x > 2$, $\log x \leq (\log2 - 1)(x-1) \leq .7 (x-1).$
Hence,
\begin{align*}
\E[\log X] = \E[\log X|X \leq 2]\Pr[X \leq 2] + \E[\log X|X > 2]\Pr[X > 2] \\
\leq \E[(X-1) -.3(X-1)^2|X \leq 2]\Pr[X \leq 2] + \E[.7(X-1)|X > 2]\Pr[X > 2]\\
= -.3 \epsilon_1 + \E[(X-1) |X \leq 2]\Pr[X \leq 2] +  .7\E[(X-1)|X > 2]\Pr[X > 2]\\
= -.3 \epsilon_1 + \E[(X-1)] - .3 \E[(X-1)| X > 2] \Pr[X > 2]\\
= -.3 \epsilon_1 - .3 \E[(X-1)| X > 2] \Pr[X > 2]\\
\leq -.3 \epsilon_1 -.3 \Pr[X > 2].
\end{align*}
If $\Pr[X > 2] = 0$, then $\epsilon_1 = \V[X] \geq \epsilon$. Else, there is at least one $\lambda_{ii} \geq \lambda_{ij} > 2$. However, by Cauchy interlacing, $\lambda_{n1} \geq 2$, which occurs with probability $\Delta p_n = p_n$. Thus, if $\Pr[X > 2] > 0$, then $\Pr[X > 2] \geq p_n$. 
\end{proof}

\begin{proof}[Proof of Lemma \ref{lem:cap_ub}] Plug in $h = I_n$. By definition, \begin{align*}
\capacity_H( T, P, Q) \leq  \frac{\det(Q, T(P))}{\det(P, I_n)} = \prod_{i:\Delta q_i \neq 0}^m \left(\det \eta_i T\left(P\right)\eta_i^\dagger\right)^{\Delta q_i}
\end{align*}
If $\lambda_{ij}$ is the $i^{th}$ eigenvalue of $\eta_i T\left(P\right)\eta_i^\dagger$, then if $T(P) = I_m$ we have $\sum_{i=1}^m \Delta q_i\sum_{j  =1}^i \lambda_{ij} = \sum_{i = 1} i \Delta q_i = \Tr Q = 1$. If $T^*(Q) = I_n$, then 
$$\sum_{i=1}^m \Delta q_i\sum_{j  =1}^i \lambda_{ij} = \sum_{i = 1}^m \Delta q_i\Tr {\eta_i} T(P) {\eta_i}^\dagger =  \sum_{i = 1}^m \Tr \Delta q_i{\eta_i}^\dagger{\eta_i} T(P)  = \Tr QT(P) = \Tr T^*(Q) P = \Tr P = 1,$$
In either case, the AM-GM inequality implies 
$$\prod_{i:\Delta q_i \neq 0}^m \left(\det \eta_i T\left(P\right)\eta_i^\dagger\right)^{\Delta q_i} = \prod_{i:\Delta q_i \neq 0}^m \left(\prod_{j = 1}^i \lambda_{ij}\right)^{\Delta q_i} \leq 1.$$

\end{proof}

\begin{proof}[Proof of Lemma \ref{lem:scaled_capacity_lower_bound}] By Lemma \ref{lem:cap_evol}, 
$$\capacity(T_1, P, Q)  = \capacity(T_{g,h}, P, Q) =  \det(Q, g^\dagger g)  \capacity(T, P, Q).$$
However, because $gT(P)g^\dagger =I $, \ref{twoside_mult} of Lemma \ref{detpx_facts} shows $\det(Q, g^\dagger g) = \det(Q, T(P))^{-1}$ and so 
$$\capacity(T_1, P, Q)= \det(Q, T(P))^{-1} \capacity(T, P, Q).$$
Thus, it is enough to bound $\det(Q, T(P))^{-1}$. One can immediately check that $T(P) \preceq m n 2^{3b} I$, so $$\det(Q, T(P))^{-1} \geq 2^{- 5bm}.$$
Multiplying the above by the bound from Theorem \ref{thm:stoc_lower_bd} implies Lemma \ref{lem:scaled_capacity_lower_bound} with some slack. \end{proof}


\end{document}